\documentclass[12pt, reqno]{amsart}
\usepackage{amsmath, amssymb, latexsym, amsthm, amscd, mathrsfs, stmaryrd}
\usepackage{amsfonts, mathrsfs}
\usepackage{mathtools}
\usepackage[title, titletoc, toc]{appendix}
\usepackage[all]{xy}
\usepackage{soul}
\usepackage{tikz}
\usetikzlibrary{arrows}
\usetikzlibrary{calc}
\usepackage{pgfplots}
\pgfplotsset{compat=1.17}
\usepackage{cite}
\usepackage{enumerate}
\usepackage{tikz-cd}
\usepackage{yhmath}
\usepackage{faktor}
\usepackage{xcolor}

\newtheoremstyle{noparens}%
{}{}%
{\itshape}{}%
{\bfseries}{.}%
{ }%
{\thmname{#1}\thmnumber{ #2}\mdseries\thmnote{ #3}}
\theoremstyle{noparens}
\newtheorem{Def}[subsubsection]{Definition}
\newtheorem{rem}[subsubsection]{Remark}
\newtheorem{prop}[subsubsection]{Proposition}
\newtheorem{thm}[subsubsection]{Theorem}
\newtheorem{lem}[subsubsection]{Lemma}
\newtheorem{cor}[subsubsection]{Corollary}

\newcommand{\mbf}{\mathbf}

\newcommand{\mscr}{\mathscr}
\newcommand{\mcal}{\mathcal}
\newcommand{\mbb}{\mathbb}
\newcommand{\mrm}{\mathrm}

\newcommand{\ro}{\mrm{ro}}
\newcommand{\co}{\mrm{co}}

\newcommand{\glln}{\mathfrak{gl}_n}
\newcommand{\slln}{\mathfrak{sl}_n}

\newcommand{\uvsln}{U_{v}(\mathfrak{sl}_n)}
\newcommand{\muvgln}{MU_{v}(\mathfrak{gl}_n)}
\newcommand{\muvsln}{MU_{v}(\mathfrak{sl}_n)}
\newcommand{\mmuvgln}{M\dot{U}_{v}(\mathfrak{gl}_n)}
\newcommand{\mmuvsln}{M\dot{U}_{v}(\mathfrak{sl}_n)}
\newcommand{\Fqd}{\mbb{F}^{d}_q}
\newcommand{\Fqdp}{\mbb{F}^{d'}_q}
\newcommand{\Fqdpp}{\mbb{F}^{d''}_q}
\newcommand{\X}{\mscr{X}_{n, d}}
\newcommand{\Xp}{\mscr{X}_{n,d'}}
\newcommand{\Xpp}{\mscr{X}_{n,d''}}
\newcommand{\MS}{\mcal{MS}_{n,d}}
\newcommand{\MSp}{\mcal{MS}_{n,d'}}

\newcommand{\qMS}{{}_{\mbb Q(v)}\mcal{MS}_{n,d}}
\newcommand{\qMSp}{{}_{\mbb Q(v)}\mcal{MS}_{n,d'}}

\newcommand{\Sd}{\mcal{S}_{n,d}}

\newcommand{\Spp}{\mcal{S}_{n,d''}}
\newcommand{\qS}{{}_{\mbb Q(v)}\mcal{S}_{n,d}}

\newcommand{\qSpp}{{}_{\mbb Q(v)}\mcal{S}_{n,d''}}
\newcommand{\fa}{\mbf a}
\newcommand{\fb}{\mbf b}
\newcommand{\faa}{\bar{\mbf a}}
\newcommand{\fbb}{\bar{\mbf b}}
\newcommand{\ii}{i_{\fb, \fa, w}}
\newcommand{\ip}{i_{\fb', \fa', w'}}
\newcommand{\ipp}{i_{\fb'', \fa''}}
\newcommand{\m}{\mathfrak{m}}
\newcommand{\lam}{\bar{\lambda}}
\newcommand{\muu}{\bar{\mu}}

\usepackage{color}
\newcommand\enc[1]{
	\tikz[baseline=(X.base)] 
	\node (X) [draw, shape=circle, inner sep=0] {$ #1 $};}

\newcommand\senc[1]{
	\tikz[baseline=(X.base)] 
	\node (X) [scale=0.7, draw, shape=circle, inner sep=0] {$ #1 $};}

\allowdisplaybreaks

\title[Positivity of the Mirabolic Quantum Groups]{Positivity of the Mirabolic Quantum Groups}
\author{Zhaobing Fan}
\address{School of Mathematical Sciences\\ Harbin Engineering University, Harbin, China 150001}
\email{fanzhaobing@hrbeu.edu.cn}
\author{Haitao Ma}
\address{School of Mathematical Sciences\\ Harbin Engineering University, Harbin, China 150001}
\email{hmamath@hrbeu.edu.cn}
\author{Zhicheng Zhang}
\address{School of Mathematical Sciences\\ Harbin Engineering University, Harbin, China 150001}
\email{zhichengzhang@hrbeu.edu.cn}

\date{}
\keywords{mirabolic flag variety, mirabolic quantum groups, canonical basis}
\subjclass{17B37, 14L35, 20G43}

\begin{document}
	\begin{abstract}
	   In this paper, we construct the canonical basis for the modified mirabolic quantum $\slln$. 
      We then study the associated $\dot{U}_v(\slln)$-module and $\dot{U}_v(\slln)$-comodule structures 
      and prove the positivity of this basis with respect to these two structures.
	\end{abstract}

\maketitle
\tableofcontents

\section{Introduction}

    The construction of canonical bases, together with the positivity phenomena they encode, is one of the central themes in the representation theory of quantum groups. A fundamental result is provided by the work of Beilinson, Lusztig, and MacPherson\cite{BLM90}. Using the geometry of pairs of partial flag varieties of type $A$, they realized the $q$-Schur algebra $\Sd$ as a convolution
    algebra. The intersection cohomology complexes of orbit closures give rise to a canonical basis of $\Sd$. More importantly, the geometric interpretation of convolution implies that the corresponding structure constants belong to $\mathbb N[v,v^{-1}]$.

    The BLM realization goes further. The multiplication formulas for the $q$-Schur algebras exhibit a remarkable stability as $d$ varies, 
    and the resulting stabilization realizes the modified quantum group $\dot{U}_v(\glln)$. It also produces a distinguished bar-invariant basis, usually referred to as the
    BLM basis. This basis retains the familiar formal features of a canonical basis, including bar-invariance and unitriangularity with respect to the standard basis. 
    Nevertheless, Li and Wang \cite{LW18} showed that the structure constants of the BLM basis of $\dot{U}_v(\glln)$ need not be positive. In fact, negative structure
    constants already occur in rank two. 

    This phenomenon is in contrast with the canonical basis of $\dot{U}_v(\slln)$. 
    Positivity of the structure constants of canonical bases of modified quantum groups belongs to the general positivity conjectures
    formulated by Lusztig \cite[25.4.2]{L93}.  In type $A$, the positivity for $\dot{U}_v(\slln)$ can be deduced from the positivity of the canonical
    bases of $q$-Schur algebras together with the compatibility of transfer maps, as developed in the work of McGerty\cite{Mc12}. 

    This foundational work \cite{BLM90} has been extended in several directions including  quantum symmetric pairs. 
    Geometric Schur dualities associated with partial flag varieties of classical type lead to Schur-type algebras and modified coideal subalgebras,
    together with their canonical and BLM bases \cite{BKLW18,FL15}.  As in type $A$, the canonical bases of $\imath$Schur algebras have positivity properties.  
    On the other hand, Li and Wang \cite{LW18} showed that the BLM basis of the modified coideal algebra of $\glln$ again has negative structure constants.  
    Motivated by this failure, they constructed instead a canonical basis for the modified coideal subalgebra of quantum $\slln$. This basis has positive structure constants and enjoys further positivity properties with respect to a geometric bilinear form and its action on tensor space.
    The positivity theory was further developed by Fan and Li \cite{FL21}. They proved the positivity of the canonical basis of modified quantum $\slln$ with respect to comultiplication and established the same positivity of the $\imath$-canonical basis.  

    The affine analogues of the BLM realization were developed independently by Ginzburg and Vasserot \cite{GV93} and by Lusztig \cite{Lu99}.
    McGerty \cite{Mc12} constructed the canonical basis, which is identified with Lusztig's canonical bases, and established the positivity of the structure constants of multiplication. 
    Positivity with respect to comultiplication was further established by Fan and Li \cite{FL21}. Fu and Shoji also proved the positivity of multiplication and comultiplication in affine type $A$ \cite{FS18, Fu21}. Recently, Fang and He developed a new realization to establish the positivity of canonical basis of the modified form of quantum groups with symmetric Cartan datum\cite{FH25}.
    The corresponding theory for affine quantum symmetric pairs was developed by
    Fan, Lai, Li, Luo, and Wang \cite{FLLLW20,FLLLW23}, who constructed BLM realizations associated with affine flag varieties and
    affine Hecke algebras of classical type, together with monomial and canonical
    bases and various positivity properties. Further developments include the multiplication formulas of Fan and Li \cite{FL19} and, more recently, the affine
    type $D$ construction of Chen, Fan and Wang \cite{CFW25,CF26}, where canonical bases for idempotented $\imath$-quantum groups and their positivity properties are obtained.

	Another natural extension in the exploration of quantum groups involves venturing into the `mirabolic' setting. 
    The mirabolic setting is of independent interest, in part because of the close relationship between $\mathscr D$-modules on $GL_d(\mathbb C)\times V$ and spherical trigonometric Cherednik algebras \cite{FG10}. Rosso initiated a mirabolic version of the BLM construction in \cite{R18}. He introduced the mirabolic $q$-Schur
    algebra $\MS$ defined as the convolution algebra of $GL_d$-invariant functions on $\X\times\X\times V$. The $GL_d$-orbits in $\X\times\X\times V$ are parametrized by decorated matrices $(A,\Delta)$, and the intersection cohomology complexes of their closures give rise to a distinguished basis $\bigl\{\{A\}_{\Delta}\bigr\}$
    of $\mathcal{MS}_{n,d}$. More recently, Goyal and Rosso \cite{GR25} constructed a $U_v(\slln)$-comodule structure on the mirabolic quantum group to study the representation theory of it, while the stabilization of the multiplication formulas for $\mathcal{MS}_{n,d}$ was established in \cite{FZM25}. The latter construction produces the modified mirabolic quantum $\glln$, denoted by $\mmuvgln$.

    However, the positivity properties of canonical bases in the mirabolic setting remain largely unexplored. In the ordinary $q$-Schur algebra, the Decomposition Theorem of perverse sheaves applied to the relevant convolution product yields positivity of multiplication in the canonical basis. In the mirabolic setting, however, the convolution
    product defining the algebra multiplication does not satisfy the condition of the Decomposition Theorem (cf.\cite{R18} for more details). 
    Consequently, the existence of an intersection cohomology basis does not by itself guarantee that its multiplication coefficients belong to $\mathbb N[v,v^{-1}]$.

    This leads to a natural question. With respect to which natural algebraic structures does the resulting basis satisfy positivity?
    Our starting point is the natural embedding
    \[
   \Sd\hookrightarrow\mathcal{MS}_{n,d}.
    \]
    Through convolution, this embedding equips $\mathcal{MS}_{n,d}$ with an $\Sd$-bimodule structure. 
    Our first result shows that the coefficients of this module action are positive with respect to the canonical bases of $\Sd$ and $\mathcal{MS}_{n,d}$. 
    \begin{prop}(Proposition \ref{pos-cb-schur})
        For any canonical basis elements $b\in \Sd$ and canonical basis elements $m\in \MS$
        $$
       b\ast m=\sum_{m'}c_{b,m}^{m'}m',
        $$
        $$
         m\ast b =\sum_{m'}c_{m, b}^{m'}m',
        $$
        where $m'$ runs over all canonical basis elements of $\MS$.
        The coefficients $c_{b,m}^{m'}$ and $c_{m, b}^{m'}\in \mathbb{N}[v, v^{-1}]$.
    \end{prop}

    In the BLM realization, a monomial basis of modified form arising from generating process and stabilization is used to construct the BLM basis (see, for example, \cite{BLM90}, \cite{BKLW18}). The generating process introduced in \cite{FZM25} does yield a monomial basis of $\mathcal{MS}_{n,d}$ and, after stabilization, a monomial basis of $\mmuvgln$. However, the coefficient of its leading term may be a nontrivial power of $v$. Consequently, although this construction still yields a monomial basis of $\mmuvgln$, it is not suitable for constructing a BLM basis because it fails to satisfy bar-invariance. 
    Using the multiplication formulas established in \cite{FZM25}, we introduce a new generating process for $\mathcal{MS}_{n,d}$ under which every monomial element has leading coefficient $1$. After stabilization, these monomial elements give rise to a suitably normalized monomial basis of $\mmuvgln$. An argument similar to that in \cite{L90} then produces a BLM basis of $\mmuvgln$. We further determine the relation between this BLM basis and the canonical bases of mirabolic $q$-Schur algebras.

    The above construction nevertheless exhibits the same phenomenon as in \cite{LW18}: the BLM basis of $\mmuvgln$ does not satisfy positivity with respect to the natural
    $\dot U_v(\glln)$-module structure. We therefore pass to the modified mirabolic quantum $\slln$, denoted by $\mmuvsln$. To construct its canonical basis, we define transfer
    maps between mirabolic $q$-Schur algebras and give a geometric construction of an inner product on $\mathcal{MS}_{n,d}$. We establish the compatibility of the canonical bases with both the transfer maps and the inner product. Relating $\mmuvsln$ to $\mmuvgln$, we obtain surjective homomorphisms
    \[
   \phi_d^{\m}:\mmuvsln\longrightarrow\mathcal{MS}_{n,d}.
    \]
    These maps allow us to construct a canonical basis $\mathbf{M\dot{B}}$ of $\mmuvsln$ which is compatible with the intersection cohomology bases of the mirabolic $q$-Schur algebras. 
    The positivity of canonical basis of $\MS$ described above then implies positivity of the coefficients of the $\dot U_v(\slln)$-module action on $\mmuvsln$.
    \begin{thm}(Theorem \ref{p-module})
        For any $b\in \mathbf{\dot{B}}$ and $m\in \mathbf{M\dot{B}}$, we write
        $$
        b\cdot m=\sum_{m'\in \mathbf{M\dot{B}}}P^{m'}_{b, m}m'.
        $$
        where $P^{m'}_{b, m}=0$ for all but finitely many $m'$.
        Then we have $P^{m'}_{b, m}\in \mathbb{N}[v,v^{-1}]$.
    \end{thm}

    We next study the corresponding comodule structure. We first geometrically construct a comodule structure on $\mathcal{MS}_{n,d}$. 
    The positivity of its coefficients with respect to the intersection cohomology bases is established by using
    Braden's hyperbolic localization theorem \cite{B03}. 
    \begin{prop}(Proposition \ref{pos-cmod})
    For canonical basis elements $m\in \MS$, and each component $\rho$ of the comodule map, we have
    $$
    \rho(m)=\sum_{m', b}c^{ b, m'}_{m}b\otimes m',
    $$
     where $b$ and $m'$ runs over all canonical basis elements of $\Spp$ and $\MSp$, respectively.
     The coefficient $c^{ b, m'}_{m} \in \mathbb{N}[v, v^{-1}]$.   
    \end{prop}
    This positivity also implies the positivity of the transfer map. 
    
    Passing to the modified setting, the $U_v(\slln)$-comodule structure on $\muvsln$ induces a $\dot U_v(\slln)$-comodule structure on $\mmuvsln$.
    We then prove the positivity of its coefficients with respect to the canonical bases via the maps $\phi_d^{\m}$.
    \begin{thm}(Theorem \ref{pos-cmod-sln})
    Let $m\in \mbf{M\dot{B}}$. For the $\dot U_v(\slln)$-comodule structure map $\rho$, we write
    $$
    \rho(m)=\sum_{b\in \mathbf{\dot{B}}, m'\in \mathbf{M\dot{B}}} C_{m}^{b,m'}b\otimes m'.
    $$
    Then $C_{m}^{b,m'}\in \mbb N[v, v^{-1}]$.
    \end{thm}

    The paper is organized as follows. In Section 2, we recall the construction of the mirabolic $q$-Schur algebra $\MS$ and its canonical basis.
    In Section 3, we construct the BLM basis of the modified mirabolic quantum $\glln$. In Section 4, we construct the canonical basis of the modified mirabolic quantum $\slln$. In Section 5, we prove the positivity properties of the canonical basis of $\mmuvsln$.
\subsection*{Acknowledgement} Z. Fan was partially supported by the National Natural Science Foundation of China (No.12631002). 
The authors used AI tools to polish the exposition, ensure consistency of notation, and assist with proofreading. 
The authors take full responsibility for the content of this paper.
\section{Mirabolic q-Schur algebra $\MS$}
\subsection{Canonical Basis of $\MS$}
Let $v$ be an indeterminate, $\mathcal{A}=\mbb Z[v,v^{-1}]$ and $\mathbb{F}_q$ be a finite field with $q$ elements. 
Let $V$ be a $d$-dimensional vector space over $\mbb F_q$ and $G=GL(V)$.

Whenever orbit closures and intersection cohomology complexes are considered, all varieties are implicitly base changed to the algebraic closure $\overline{\mathbb{F}}_q$ of $\mathbb{F}_q$.

For a fixed positive integer $n$, let $\mathscr{X}_{n,d}$ be the set of all $n$-step partial flags in $V$, i.e., 
$$
\mathscr{X}_{n,d}=\{f=(0=V_0\subseteq V_1\subseteq \cdots \subseteq V_{n-1}\subset V_n=V)\}.
$$

Recall the definition of decorated matrices.
\begin{Def}\label{dec mat}
  A decorated matrix is a pair $(A, \Delta)$, where $A=(a_{i j})\in Mat_{n\times n}(\mathbb{N})$ 
  and $\Delta =\{(i_1, j_1), \cdots, (i_k, j_k)\}$ is a set that satisfies
  $$
  1\leq  i_1< \cdots < i_k \leq  n , \ 1\leq j_k< \cdots< j_1 \leq n, 
  $$
  with the additional condition that for all $(i, j)\in \Delta$, the entry $a_{i j}>0$. We denote by $\Xi_{n|n}$ the set of $n\times n$ decorated matrices and 
  $\Xi_{n|n,d}$ is defined as $\{(A, \Delta)\in \Xi_{n|n} | \sum_{i j}a_{i, j}=d \}$.
\end{Def}

By \cite{MWZ}, we have the bijection between the set of $G$-orbits of $\mathscr{X} \times \mathscr{X}\times V$ and $\Xi_{n|n,d}$, i.e., 
$$
G \backslash  \mathscr{X}_{n,d}\times \mathscr{X}_{n,d}\times V \longleftrightarrow \Xi_{n|n,d}
.$$
We denote by $\mathcal{O}_{A, \Delta}$ the orbit corresponding to $(A, \Delta)$.

The mirabolic $q$-Schur algebra $\mcal{MS}_{n, d}=\mcal A_{G}(\mscr X_{n, d}\times \mscr X_{n,d}\times V)$ is a convolution algebra with the following convolution product:
For any $g,h\in \mathcal{MS}_{n, d}$,
$$
h\ast g(f,f',w)=\sum_{(f'',\mu)\in \mathscr{X}_{n,d}\times V}h(f,f'',\mu)g(f'',f',w-\mu)
.$$
The characteristic functions of the $G$-orbits form an $\mcal A$-basis of $\MS$: $\{e_{A, \Delta}\ |\ (A, \Delta)\in \Xi_{n|n,d} \}$, where $e_{A, \Delta}$ is the characteristic function of the orbit $\mathcal{O}_{A, \Delta}$. 

Take specilization $v^2=q$ and let $d(A, \Delta)=\dim \mathcal{O}_{A, \Delta}$ and $r(A, \Delta)$ denote the dimension of the image of $\mathcal{O}_{A, \Delta}$ under the first projection $pr_1: \mathscr{X}\times \mathscr{X}\times V\to \mathscr{X}$. Then we define $[A]_{\Delta}=v^{-d(A,\Delta)+r(A, \Delta)}e_{A,\Delta}$.

We define a partial order on $\Xi_{n|n, d}$ by $(A,\Delta)\preceq (A', \Delta')$ if $\mcal O_{A, \Delta}\subset \overline{\mcal O_{A', \Delta'}}$ and $(A,\Delta)\prec (A', \Delta')$ if $(A,\Delta)\preceq (A', \Delta')$ and $(A,\Delta)\neq (A', \Delta')$. 
Consider the intersection cohomology complex $IC_{A, \Delta}$ of the closure $\overline{\mcal O_{A, \Delta}}$. Then we define the polynomial in $\mathcal{A}$
$$
P_{(A, \Delta), (A', \Delta')}=\sum_{i}\mathrm{rank} H^{i}(IC_{A,\Delta}|_{\mcal{O}_{A', \Delta'}})v^{i-d(A,\Delta)+d(A',\Delta')}
$$
By definition, we can see 
\begin{itemize}
    \item $P_{(A, \Delta), (A', \Delta')}=0$ unless $(A',\Delta')\preceq (A,\Delta)$;
    \item $P_{(A,\Delta), (A,\Delta)}=1$;
    \item $P_{(A, \Delta), (A', \Delta')}\in v^{-1}\mathbb{Z}[v^{-1}]$ for any $(A', \Delta')\prec (A,\Delta)$.
\end{itemize}

Consider new elements
$$\{A\}_{\Delta}=\sum_{(A',\Delta')\preceq (A,\Delta)}P_{(A, \Delta), (A', \Delta')}[A']_{\Delta'}.$$
Then we can see that $\{\{A\}_{\Delta}\ |\ (A,\Delta)\in \Xi_{n|n, d}\}$ forms a $\mathcal{A}$-basis of $\MS$ and we call it a canonical basis of $\MS$.

Besides, the Verdier duality implies that there is a unique homomorphism of abelian groups $^-:\mathcal{MS}_{n, d}\to \MS$ such that 
$\bar{v}=v^{-1}$ and $\overline{x\{A\}_{\Delta}}=\overline{x}\{A\}_{\Delta}$ for all $(A,\Delta)\in \Xi_{n|n,d}$ and $x\in \mcal A$. 

Recall in \cite{BLM90}, the $q$-Schur algebra $\Sd$ can be realized as a convolution algebra on $\mcal A_{G}(\X\times \X)$. 
The basis of $\Sd$ can be indexed by $A\in \Theta_d$, where $\Theta_d=\{A\in Mat_{n\times n}(\mbb N)\mid \sum_{i,j}a_{ij}=d\}$.
By this realization, the algebra $\Sd$ is naturally a subalgebra of $\MS$ by regarding $A\in \Theta_d$ as $(A,\emptyset)\in \Xi_{n|n,d}$.
Therefore for any $h\in \Sd$, $g\in \MS$ and $(f, f',w)\in \X\times \X\times V$, the convolutions
\begin{align}
    h\ast g(f,f',w)&=\sum_{f''\in \mathscr{X}_{n,d}}h(f,f'')g(f'',f',w)\\
    g\ast h(f,f',w)&=\sum_{f''\in \mathscr{X}_{n,d}}g(f,f'',w)h(f'',f')
\end{align}
endow $\MS$ with a $\Sd$-bimodule structure.

By \cite{BLM90}, the algebra $\Sd$ admits the canonical basis $\{\{A\}\mid A\in \Theta_d\}$ and a bar involution defined by $\overline{x\{A\}}=\bar{x}\{A\}$. 
By the construction of the canonical basis of $\Sd$ and $\MS$, $\bar{}$ is a $\Sd$-bimodule homomorphism on $\MS$.

\section{The BLM basis of $\mmuvgln$}
\subsection{Monomial Basis of $MS$}

    We recall a partial order on $\Xi_{n|n,d}$ introduced  in \cite{FZM25}. We begin with the partial order on the $\Delta$.
    \begin{Def}
        We write $\Delta' \leq \Delta$ if, for every $(i, j)\in \Delta'$, there exists $(k,l)\in \Delta$ such that $i\leq k$ and $j\leq l$.
        We write $\Delta' < \Delta$ if $\Delta' \leq \Delta$ and $\Delta'\neq \Delta$.
    \end{Def} 

    For an $(A,\Delta)\in \Xi_{n|n,d}$, define its row-sum and column-sum vectors by
	$$
	\ro(A)=(\sum\limits_{1\leq j \leq n}a_{1j}, \cdots, \sum\limits_{1\leq j \leq n}a_{nj}),
	\ 
	\co(A)=(\sum\limits_{1\leq i \leq n}a_{i1}, \cdots, \sum\limits_{1\leq i \leq n}a_{in}).
	$$   
    Thus $\ro(A)_i=\sum\limits_{1\leq j \leq n}a_{ij}$ and $\co(A)_j=\sum\limits_{1\leq i \leq n}a_{ij}$.

    Recall the Bruhat order $\leq $ on $\Theta_d$ in \cite[3.5]{BLM90}, we define the partial order on $\Xi_{n|n, d}$.
  \begin{Def}
    Let $(A, \Delta), (A', \Delta')\in \Xi_{n|n,d}$, and suppose $\co(A)=\co(A')$, $\ro(A)=\ro(A')$. 
    We write $(A', \Delta')\leq (A,\Delta)$ if $A'< A$ or $A'=A$ and $\Delta'\leq \Delta$.
    We write $(A', \Delta')<(A, \Delta)$ if $(A', \Delta')\leq (A, \Delta)$ and $(A',\Delta') \neq  (A,\Delta)$.
  \end{Def}

  We denote by ``$[A]_{\Delta}+\text{ lower terms}$'' an element of $\MS$ 
  which equals $[A]_{\Delta}$ plus a $\mathcal{A}$-linear combination of elements $[A']_{\Delta'}$ with $(A',\Delta')\in \Xi_{n|n, d}$ and $(A',\Delta')<(A,\Delta)$.

    For a diagonal matrix $D=\operatorname{diag}(d_1,\ldots,d_n)\in \Theta_d$ and a integer $1\le h\le n$, 
    such that $(D,\{(h,h)\})\in\Xi_{n|n,d}$, we define
    \[
    L_{D,h}=
    \sum_{i=1}^{h}
    v^{-\sum_{k=i+1}^{h}d_k}
    [D]_{\{(i,i)\}}
    +
    v^{-\sum_{i=1}^{h}d_i}
    [D]_{\emptyset}.
    \]
    Here $[D]_{\{(i,i)\}}=0$ whenever $d_i=0$.


\begin{lem}\label{leading}
Assume $1\leq h< n$, $1\leq r\leq n$. Let $(A, \Delta)\in \Xi_{n|n,d}$, where $\Delta=\{(i_1, j_1), \cdots, (i_k, j_k)\}$. 
Define
\[
j_*=
\begin{cases}
j_2,& k\geq2,\\
0,& k=1,
\end{cases}
\qquad
\mathcal J
 =
 \{t\in\mathbb N\mid j_*<t<r,\ a_{h+1,t}>0\}.
\]
If $\mcal J\neq \emptyset$, we set $l=\max \mcal J$.
\begin{enumerate}[ $(a)$ ]
    \item Suppose $(A, \Delta)$ satisfies the following conditions:
    $$a_{hj}=0, \forall j\geq r;\quad a_{h+1, r}=R;\quad a_{h+1, j}=0, \forall j>r$$
    and
    $$(i_1, j_1)=(h+1, r).$$
    Let $(B,\emptyset)\in \Xi_{n|n,d}$ such that $B-RE_{h, h+1}$ is a diagonal matrix and $\co(B)=\ro (A)$. Then
    \begin{align*}
        [B]_{\emptyset}\ast [A]_{\Delta}=[M]_{\Delta'}+ \text{lower terms},
    \end{align*}
    where $M=A+R(E_{h,r}-E_{h+1,r})$ and 
    $$\Delta'=
    \begin{cases}
        (\Delta\setminus \{(h+1,r)\})\cup\{(h,r)\} & \text{ if } \mcal J=\emptyset;\\
        (\Delta\setminus \{(h+1,r)\})\cup\{(h,r), (h+1,l)\}& \text{ if } \mcal J\neq \emptyset.
    \end{cases}
   $$
     \item Suppose $(A, \emptyset)\in \Xi_{n|n,d}$, satisfies the following conditions:
    $$a_{ij}=0, \forall i\leq h, j>r, a_{hr}\neq 0.$$
    Let $(D,\{(h, h)\}), (A, \emptyset)\in \Xi_{n|n,d}$, such that $D$ is a diagonal matrix and $\co(D)=\ro (A)$. Then
    \begin{align*}
        [D]_{\{(h, h)\}}\ast [A]_{\emptyset}=[A]_{\{(h, r)\}}+ \text{lower terms};
    \end{align*}
    \item Suppose $(A, \Delta)\in \Xi_{n|n,d}$, where $\Delta=\{(i_1, j_1), \cdots, (i_k, j_k)\}$, satisfies the following conditions:
    $$a_{hj}=0, \forall j< r;\quad a_{h, r}=R;\quad a_{h+1, j}=0, \forall j\leq r$$
    and 
    $$i_k=h,\quad j_k\geq r.$$
    Let $(C,\emptyset)\in \Xi_{n|n,d}$ such that $C-RE_{h+1, h}$ is a diagonal matrix and $\co(C)=\ro (A)$. Then
    \begin{align*}
        [C]_{\emptyset}\ast [A]_{\Delta}=[M]_{\Delta'}+ \text{lower terms},
    \end{align*}
    where $M=A-R(E_{h,r}-E_{h+1,r})$ and 
    $$
    \Delta'=\begin{cases}
        \Delta\cup\{(h+1,r)\} & \text{ if } j_k>r;\\
        (\Delta\setminus \{(h,r)\})\cup\{(h+1,r)\}& \text{ if } j_k=r.
    \end{cases}
    $$
\end{enumerate}
\end{lem}
\begin{proof}
Let $\mathbf e_r$ denote the $r$-th standard basis vector of
$\mathbb N^n$. The argument in \cite[Lemma 3.8]{BLM90} shows that in $(a)$ and $(c)$, every summand indexed by
$\mbf p\neq R\mbf e_r$ is strictly smaller than $M$. It remains to determine the maximal $\Delta'$ of $M$ and compute its coefficient.

For part $(a)$, apply \cite[Proposition 4.1.7 $(c)$]{FZM25}. More precisely, for $\mathbf{p}\in P_R=\{\mathbf{p}=(p_{1},\cdots, p_n)\in \mathbb{N}^n\ |\ \sum_{i}p_i=R \}$, 
let $\kappa_{A}(\mathbf{p})=\sum_{j\geq u}a_{h, j}p_u-\sum_{j>u}a_{h+1, j}p_u+\sum_{u<u'}p_{u}p_{u'}$ and $X_{\mathbf{p}}=A+\sum_{u}p_u(E_{h, u}-E_{h+1, u})$
    \begin{align*}
        & [B]_\emptyset\ast [A]_{\Delta}\notag\\
        = &\sum\limits_{\mathbf{p}\in P_R, a_{h+1, u} \geq p_u}
  v^{\kappa_{A}(\mathbf{p})}\prod^{n}_{u=1}\overline{[a_{h, u}+p_u, p_u]}_v[X_{\mathbf{p}}]_{\Delta} \\
       & + \sum\limits_{\substack{\mathbf{p}\in P_R, p_{r}>0,\\ a_{h+1, u} \geq p_u}} v^{\kappa_{A}(\mathbf{p})-\sum_{1< j \leq r}(a_{h+1, j}-p_j)}\prod^{n}_{u=1}\overline{[a_{h, u}+p_u-\delta_{u,r}, p_u-\delta_{u,r}]}_v[X_{\mathbf{p}}]_{\Delta_c}\\
       & +\sum\limits_{\substack{\mathbf{p}\in P_R, a_{h+1, u} \geq p_u,\\ p_{r}>0, j_{\ast}\leq t< r}}v^{\kappa_{A}(\mathbf{p})-\sum_{t< j \leq r}(a_{h+1, j}-p_j)}\prod^{n}_{u=1}\overline{[a_{h, u}+p_u-\delta_{u,r}, p_u-\delta_{u,r}]}_v[X_{\mathbf{p}}]_{\Delta_t},
    \end{align*}
  where
  \begin{align*}
  \Delta_c&= \{(h, r), (i_2, j_2), \cdots, (i_k, j_k)\},\\
  \Delta_t& = \{(h, r), (h+1, t), (i_{2}, j_{2}), \cdots, (i_k, j_k)\},
  \end{align*}
  and $$[N, t]_v=\prod_{i\in [1, t]}\frac{v^{2(N-i+1)}-1}{v^{2i}-1}.$$
  For $\mathbf{p}=R\mbf e_r$, the corresponding $X_{\mbf p}=M$. The term carrying the original decoration $\Delta$ is absent, because 
  $M_{h+1,r}=0$ and hence $(M,\Delta)$ is not a decorated matrix.

The second sum contributes the $\Delta_c$. The third sum contributes
\[
\Delta_t=\Delta_c\cup\{(h+1,t)\}
\]
for those $t\in\mathcal J$. We have
\[
\Delta_c<\Delta_t,
\]
and, whenever $t<t'$,
\[
\Delta_t<\Delta_{t'}.
\]
Consequently, the maximal $\Delta'$ is $\Delta_c$ if $\mathcal J=\emptyset$, and it is $\Delta_l$ if $\mcal J\neq \emptyset$. It remains to compute the coefficient. Under the assumptions of
part $(a)$,  $\overline{[a_{h, u}+p_u-\delta_{u,r}, p_u-\delta_{u,r}]}_v=1$ for each $1\leq u\leq n$ and
\[
\kappa_A(R\mathbf e_r)
 =
 R\left(
 \sum_{j\geq r}a_{h,j}
 -
 \sum_{j>r}a_{h+1,j}
 \right)
 =0.
\]
If $\mathcal J=\emptyset$, $\sum_{1< j \leq r}(a_{h+1, j}-p_j)=0$, so the coefficient of $[M]_{\Delta_c}$ is $1$.
If $\mathcal J\neq\emptyset$, $\sum_{l< j \leq r}(a_{h+1, j}-p_j)=0$, so the coefficient of $[M]_{\Delta_l}$ is $1$.

  For $(b)$, recall the formula \cite[Proposition 4.1.7 $(d)$]{FZM25}, 
  the assumptions imply that no term with $t>r$ occurs. 
  We have
   \begin{align*}
		[D]_{\{(h,h)\}}\ast [A]_{\emptyset}=v^{-\sum_{i\leq h,j>r}a_{i, j}} [A]_{\{(h,r)\}}+\text{ lower terms}.
    \end{align*}
    The exponent $\sum_{i\leq h, j>r}a_{ij}$ is zero, so the coefficient of 
    $[A]_{\{(h,r)\}}$ is $1$. This proves part $(b)$.

For part $(c)$, apply proposition \cite[Proposition 4.1.8 $(b)$]{FZM25}. 
More precisely, for any $\mathbf{p}\in P_R$, 
let $\kappa_{A}'(\mathbf{p})=\sum_{j\leq u}a_{h+1, j}p_u-\sum_{j<u}a_{h, j}p_u+\sum_{u<u'}p_{u}p_{u'}$ and $X'_{\mathbf{p}}=A+\sum_{u}p_u(E_{h+1, u}-E_{h, u})$ 
we have 
    \begin{align*}
        & [C]_\emptyset\ast [A]_{\Delta}\\
    = & \sum\limits_{\mathbf{p}\in P_R, a_{h, u} \geq p_u} v^{\kappa_{A}'(\mathbf{p})-\sum_{1<u\leq j_k}p_u}\prod^{n}_{u=1}\overline{[a_{h+1, u}+p_u, p_u]}_v[X'_{\mathbf{p}}]_{\Delta} \\
      & +\sum\limits_{\substack{\mathbf{p}\in P_R, a_{h, u} \geq p_u\\p_{j_k}>0}} v^{\kappa_{A}'(\mathbf{p})-\sum_{1\leq j\leq j_k}a_{h+1,j}} \prod^{n}_{u=1}\overline{[a_{h+1, u}+p_u-\delta_{u, j_k}, p_u-\delta_{u, j_k}]}_v[X'_{\mathbf{p}}]_{\Delta_b} \\  
      & + \sum\limits_{\substack{\mathbf{p}\in P_R, a_{h, u} \geq p_u\\ 1\leq t<j_k, p_t>0}}v^{\kappa_{A}'(\mathbf{p})-\sum_{1\leq j\leq t}a_{h+1,j}}\prod^{n}_{u=1}\overline{[a_{h+1, u}+p_u-\delta_{u, t}, p_u-\delta_{u, t}]}_v[X'_{\mathbf{p}}]_{\Delta_t},
    \end{align*}
  where
  \begin{align*}
  \Delta_b &= \{(i_1, j_1), \cdots, (i_{k-1}, j_{k-1}), (h+1, j_k)\}, \\
  \Delta_t &= \{(i_1, j_1), \cdots, (i_{k-1}, j_{k-1}), (h, j_k), (h+1, t)\}.
  \end{align*}

  For $\mathbf p=R\mathbf e_r$, we have
    \[
    X'_{\mathbf p}=M
    \]
    and
    \[
    \kappa'_A(R\mathbf e_r)
     =
     R\left(
     \sum_{j\leq r}a_{h+1,j}
     -
     \sum_{j<r}a_{h,j}
     \right)
     =0.
    \]
 The terms carrying the original $\Delta$ are lower term. 
 For the second sum, if $j_k=r$ it contributes precisely
\[
[M]_{(\Delta\setminus\{(h,r)\})\cup\{(h+1,r)\}}.
\]
If $j_k>r$,  it contributes lower term for $\mbf p=R\mbf e_r$, 
while the third sum, with $t=r$, only contributes
\[
[M]_{\Delta\cup\{(h+1,r)\}}.
\]
when $j_k>r$.
Since
\[
\Delta<\Delta\cup\{(h+1,t)\}<(\Delta\setminus \{(h, r)\})\cup\{(h+1,r)\},
\]
We have the desired leading term for both cases.
Now under the assumption of $(c)$, the coefficients are $1$ by calculations similar to those in $(a)$.
\end{proof}

For any $[A]_{\Delta}\in \MS$, where $\Delta=\{(i_1, j_1), \cdots, (i_k, j_k)\}$, 
set
\begin{align}\label{circ start position}
    \vartheta(\Delta)=\max_{(i, j)\in \Delta} \{\min \{i,j\}\},\quad \vartheta(\emptyset)=0.
\end{align}

Define two subsets of $\Delta$
    \begin{align*}
    \Delta_{+}&= \{(i,j)\in \Delta \mid i\leq \vartheta(\Delta)\leq j\}, \\
    \Delta_{-}&= \{(i,j)\in \Delta\mid j\leq \vartheta(\Delta) <i\}.
\end{align*}
Moreover, we associate these data with four subsets of $[1,n]\times [1,n]$, where $[1,n]$ means the set of integers $\{1,\cdots, n\}$:
\begin{align*}
    I^{\leq}_{+}=&\{(i, j) \mid \{(i, j)\}\leq \Delta_{+}, (i, j)\neq (i_1, j_1)\};\\
    I^{>}_{+}=&\{(i, j)\ \mid 1\leq i<j\leq n\}\setminus  (I^{\leq}_{+}\cup\{(i_1, j_1)\});\\
    I^{\leq}_{-}=&\{(i, j) \mid \{(i, j)\}\leq \Delta_{-}\}.\\
    I^{>}_{-}=&\{(i, j)\ \mid 1\leq j<i\leq n\}\setminus  I^{\leq}_{-};
\end{align*}

Define a total order on the set
$$\{(i,h,j)\mid 1\leq i, j\leq n, 1\leq h\leq n-1\}$$ as follows:
$$(i, h, j)<(i', h',j')$$ 
if and only if
$$
i>i'\text{ or } i=i', j>j' \text{ or } i=i', j=j', h<h'.
$$

We define the following monomials
\begin{align}
    &[A]_{\Delta_{-}}^{>}=\prod^{\longleftarrow}_{\substack{(i,j)\in I^{>}_{-}\\ 1\leq j\leq h<i\leq n}}[D_{i, h, j}+a_{ij}E_{h+1, h}]_{\emptyset}\label{lr};\\
    &[A]_{\Delta_{-}}^{\leq }=\prod^{\longleftarrow}_{\substack{(i,j)\in I^{\leq }_{-}\\ 1\leq j\leq h<i\leq i_k}}[D_{i, h, j}+a_{ij}E_{h+1, h}]_{\emptyset}\label{ll};\\
    &[A]_{\Delta_{+}}^{>}=\prod^{\longrightarrow}_{\substack{(i,j)\in I^{>}_{+}\\1\leq i\leq h<j\leq n}}[D_{i, h, j}+a_{ij}E_{h, h+1}]_{\emptyset}\label{ur};\\
    &[A]_{\Delta_{+}}^{\leq}=\prod^{\longrightarrow}_{\substack{(i,j)\in I^{\leq}_{+}\\ 1\leq i\leq h<j\leq n}}[D_{i, h, j}+a_{ij}E_{h, h+1}]_{\emptyset}.\label{ul}
\end{align}
All products below are taken with respect to the multiplication
$\ast$.
The symbols 
$$\prod^{\longleftarrow},\quad \prod^{\longrightarrow}$$
indicate that the factors are arranged from left to right in decreasing and increasing order with respect to the order defined on 
$\{(i,h,j)\mid 1\leq i, j\leq n, 1\leq h\leq n-1\}$, respectively.
The diagonal matrices $D_{i, h, j}\in Mat_{n\times n}(\mathbb{N})$ are uniquely determined by $\ro(A)$ and $\co(A)$.

\begin{prop}\label{mon-b}
Let $[A]_{\Delta}\in \MS$, where $\Delta=\{(i_1, j_1), \cdots, (i_k, j_k)\}$ and $\vartheta(\Delta)$ is defined in $\eqref{circ start position}$.
The following identities hold.
    \begin{enumerate}[ $(a)$ ]
        \item Suppose that $i_1\leq j_1$. Then 
          \begin{align}
        &[A]_{\Delta_{-}}^{\leq}\ast [A]_{\Delta_{+}}^{>}\ast \prod^{\longrightarrow}_{h\in [i_1, \vartheta(\Delta)-1]}[D_{i_1, h, j_1}+a_{i_1,j_1}E_{h, h+1}]_{\emptyset}\notag
        \\
        &\ast L_{D, {\vartheta(\Delta)}}\ast \prod^{\longrightarrow}_{h\in [\vartheta(\Delta), j_1-1]}[D_{i_1, h, j_1}+a_{i_1,j_1}E_{h, h+1}]_{\emptyset}\ast [A]_{\Delta_{+}}^{\leq}\ast [A]_{\Delta_{-}}^{>}\notag\\
        =& [A]_{\Delta}+\text{ lower terms}\label{mon-u};
            \end{align}
        where diagonal matrices $D_{i, h, j}\in Mat_{n\times n}(\mbb N)$ are uniquely determined, and $(D, \{(\vartheta(\Delta), \vartheta(\Delta))\})$ is a suitable element in $\Xi_{n|n,d}$.
        \item Suppose that $i_1>j_1$. Then
         \begin{align}
        &\prod^{\longrightarrow}_{1\leq i\leq h<j\leq n}[D_{i, h, j}+a_{ij}E_{h, h+1}]_{\emptyset}\ast [A]_{\Delta_{-}}^{\leq}\ast L_{D, \vartheta(\Delta)}\ast [A]_{\Delta_{-}}^{>}\notag\\
        &= [A]_{\Delta}+\text{ lower terms}\label{mon-l}. 
        \end{align}
         where diagonal matrices $D_{i, h, j}\in Mat_{n\times n}(\mbb N)$ are uniquely determined, and $(D, \{(\vartheta(\Delta), \vartheta(\Delta))\})$ is a suitable element in $\Xi_{n|n,d}$.
         \item When $\Delta=\emptyset$. Then 
        \begin{align}
        &\prod^{\longrightarrow}_{1\leq i\leq h<j\leq n}[D_{i, h, j}+a_{ij}E_{h, h+1}]_{\emptyset}\ast \prod^{\longleftarrow}_{1\leq j\leq h<i\leq n}[D_{i, h, j}+a_{ij}E_{h, h+1}]_{\emptyset}\notag\\
        &= [A]_{\Delta}+\text{ lower terms}\label{mon-kong}. 
        \end{align}
         where diagonal matrices $D_{i, h, j}\in Mat_{n\times n}(\mbb N)$ are uniquely determined, and $(D, \{(\vartheta(\Delta), \vartheta(\Delta))\})$ is a suitable element in $\Xi_{n|n,d}$.
         
    \end{enumerate}
  We denote the left-hand side of \eqref{mon-u}, \eqref{mon-l} and \eqref{mon-kong} as appropriate, by $m_{A,\Delta}$.
\end{prop}

\begin{proof}
    In the proof of \cite[Proposition 3.9]{BLM90}, the entries below the diagonal are filled starting with the bottom rows and proceeding upward, and, within each row, from right to left. The entries above the diagonal are then filled starting with the top rows and proceeding downward, and, within each row, from left to right. At each stage, the resulting leading term has coefficient $1$.
    For $(c)$, it follows \cite[Proposition 3.9]{BLM90}.
    For $(a)$ and $(b)$, we need some modification of this procedure. 
    
    Suppose first that $i_1\leq j_1$. We begin with the rightmost product,
    \[
        [A]_{\Delta_+}^{\leq}\ast[A]_{\Delta_-}^{>}.
    \]
    The ordinary BLM triangular relation gives an undecorated leading term with coefficient $1$. The factors
    \[
   \prod^{\longrightarrow}_{h\in [\vartheta(\Delta), j_1-1]}[D_{i_1, h, j_1}+a_{i_1,j_1}E_{h, h+1}]_{\emptyset},
    \]
    then move the amount $a_{i_1,j_1}$ upward along the $j_1$-th column, from the diagonal position $(j_1,j_1)$ to the position $(\vartheta(\Delta),j_1)$.
    The remaining part follows the same procedure in \cite[Proposition]{BLM90}.

    We next multiply by $L_{D,\vartheta(\Delta)}$. By Lemma \ref{leading} $(b)$, the resulting leading term carries the decoration $(\vartheta(\Delta),j_1)$.

    The factors
        \[
        \prod^{\longrightarrow}_{h\in [i_1, \vartheta(\Delta)-1]}[D_{i_1, h, j_1}+a_{i_1,j_1}E_{h, h+1}]_{\emptyset}
        \]
    move this decorated entry further upward until it reaches $(i_1,j_1)$. 
    Whenever the moving entry passes through the row $i_t\leq \vartheta(\Delta)$, with $2\leq t$, the ordering of the
    previously filled entries ensures that $j_t$ is the largest occupied column of that row strictly to the left of $j_1$. Repeated application of Lemma~\ref{leading}\textup{(a)} therefore leaves a decoration at $(i_t,j_t)$. 

    Finally, we apply the factors in $[A]_{\Delta_+}^{>}$ and then those in $[A]_{\Delta_-}^{\leq}$. The underlying matrix is completed by the ordinary BLM procedure. At the distinguished lower-triangular steps, Lemma~\ref{leading}\textup{(c)} transports the terminal decoration or creates the next decoration in
    $\Delta_-$. Thus the final leading term is $[A]_{\Delta}$, with coefficient $1$.
    
   Now suppose that $i_1>j_1$. Then $\Delta_+=\emptyset$ and $\Delta_-=\Delta$. We first apply $[A]_{\Delta_-}^{>}$. The BLM procedure produces an undecorated leading matrix in which the amount $a_{i_1,j_1}$ is still located at the diagonal position $(j_1,j_1)$. In this case, we have $\vartheta(\Delta)=j_1$. Then a decoration at $(j_1,j_1)$ arises by multiplying $L_{D,\vartheta(\Delta)}$.

    The factors in $[A]_{\Delta_-}^{\leq}$ then move the decorated entry downward from $(j_1,j_1)$ to $(i_1,j_1)$. Repeated applications of
    Lemma~\ref{leading}\textup{(c)} produce, in the prescribed order, all the remaining decorations in $\Delta_-$. The leftmost product fills the entries above the diagonal by the ordinary BLM procedure. Hence the final leading term is again $[A]_{\Delta}$ with coefficient $1$.
    
    We will explain the proof in more detail via the following example. 
    Recall $\co(A)_k=\sum^n_{i=1}a_{i,k}$ and $\ro(A)_k=\sum^n_{j=1}a_{k,j}$ where $k\in [1,n]$.
    
    Assume $n=5$ and $\Delta=\{(2, 4), (3, 3), (4, 2), (5, 1)\}$, i.e., 
  $$
  (A, \Delta)=
  \begin{pmatrix}
    a_{11} & a_{12}& a_{13} & a_{14}& a_{15}\\
    a_{21} & a_{22} & a_{23} &\enc{a_{24}} & a_{25}\\
     a_{31} & a_{32} & \enc{a_{33}} &a_{34} & a_{35}\\
    a_{41} & \enc{a_{42}} & a_{43} & a_{44}& a_{45}\\
     \enc{a_{51}} & a_{52} & a_{53} &a_{54} & a_{55}\\
  \end{pmatrix}.
  $$
  Then $\Delta_{+}=\{(2, 4), (3,3) \}$, $\Delta_{-}=\{(4,2), (5,1)\}$ and $\vartheta(\Delta)=3$. Then we first fill the elements $a_{ij}$ that
  $$
 (i, j)\in \{(5, 4), (5,3), (5,2), (4,3), (1,2), (1, 3), (2,3), (1,4)\}.
  $$
  The corresponding elements are generated by
  \begin{flalign*}
    &(B_1, \emptyset)=
    \begin{psmallmatrix}
        \co(A)_1+a_{12}+a_{13} & a_{14} & 0 & 0 & 0\\
        0 & \co(A)_2-a_{12}-a_{52}+a_{23} & 0 & 0 & 0\\
        0 & 0 & a_{33} & 0 & 0\\
        0 & 0 & 0 & \co(A)_4-a_{54}-a_{14}+a_{43} & 0\\
        0 & 0 & 0 & 0 & \co(A)_5+a_{54}+a_{53}+a_{52}
    \end{psmallmatrix}, \\
    &(B_2, \emptyset)=
    \begin{psmallmatrix}
        \co(A)_1+a_{12}+a_{13} & 0 & 0 & 0 & 0\\
        0 & \co(A)_2-a_{12}+a_{23}-a_{52} & a_{14} & 0 & 0\\
        0 & 0 & a_{33} & 0 & 0\\
        0 & 0 & 0 & \co(A)_4-a_{54}-a_{14}+a_{43} & 0\\
        0 & 0 & 0 & 0 & \co(A)_5+a_{54}+a_{53}+a_{52}
    \end{psmallmatrix}, \\
    &(B_3, \emptyset)=
    \begin{psmallmatrix}
        \co(A)_1+a_{12}+a_{13} & 0 & 0 & 0 & 0\\
        0 & \co(A)_2-a_{12}-a_{52}+a_{23} & 0 & 0 & 0\\
        0 & 0 & a_{33} & a_{14} & 0\\
        0 & 0 & 0 & \co(A)_4-a_{54}-a_{14}+a_{43} & 0\\
        0 & 0 & 0 & 0 & \co(A)_5+a_{54}+a_{53}+a_{52}
    \end{psmallmatrix}, \\
    &(B_4, \emptyset)=
    \begin{psmallmatrix}
        \co(A)_1+a_{12}+a_{13} & 0 & 0 & 0 & 0\\
        0 & \co(A)_2-a_{12}-a_{52} & a_{23} & 0 & 0\\
        0 & 0 & a_{33} & 0 & 0\\
        0 & 0 & 0 & \co(A)_4-a_{54}+a_{43} & 0\\
        0 & 0 & 0 & 0 & \co(A)_5+a_{54}+a_{53}+a_{52}
    \end{psmallmatrix}, \\
    &(B_5, \emptyset)=
    \begin{psmallmatrix}
        \co(A)_1+a_{12} & a_{13} & 0 & 0 & 0\\
        0 & \co(A)_2-a_{12}-a_{52} & 0 & 0 & 0\\
        0 & 0 & a_{23}+a_{33} & 0 & 0\\
        0 & 0 & 0 & \co(A)_4-a_{54}+a_{43} & 0\\
        0 & 0 & 0 & 0 & \co(A)_5+a_{54}+a_{53}+a_{52}
    \end{psmallmatrix}, \\
    &(B_{6}, \emptyset)=
    \begin{psmallmatrix}
        \co(A)_1+a_{12} & 0 & 0 & 0 & 0\\
        0 & \co(A)_2-a_{12}-a_{52} & a_{13} & 0 & 0\\
        0 & 0 & a_{23}+a_{33} & 0 & 0\\
        0 & 0 & 0 & \co(A)_4-a_{54}+a_{43} & 0\\
        0 & 0 & 0 & 0 & \co(A)_5+a_{54}+a_{53}+a_{52}
    \end{psmallmatrix}, \\
     &(B_{7}, \emptyset)=
    \begin{psmallmatrix}
        \co(A)_1 & a_{12} & 0 & 0 & 0\\
        0 & \co(A)_2-a_{12}-a_{52} & 0 & 0 & 0\\
        0 & 0 & \co(A)_3-a_{53}-a_{43} & 0 & 0\\
        0 & 0 & 0 & \co(A)_4-a_{54}+a_{43} & 0\\
        0 & 0 & 0 & 0 & \co(A)_5+a_{54}+a_{53}+a_{52}
    \end{psmallmatrix}, \\
   &(B_{8}, \emptyset)=
    \begin{psmallmatrix}
        \co(A)_1 & 0 & 0 & 0 & 0\\
        0 & \co(A)_2-a_{52} & 0 & 0 & 0\\
        0 & 0 & \co(A)_3-a_{53}-a_{43} & 0 & 0\\
        0 & 0 & 0 & \co(A)_4-a_{54}+a_{43} & 0\\
        0 & 0 & 0 & a_{52} & \co(A)_5+a_{54}+a_{53}
    \end{psmallmatrix}, \\
       &(B_{9}, \emptyset)=
    \begin{psmallmatrix}
        \co(A)_1 & 0 & 0 & 0 & 0\\
        0 & \co(A)_2-a_{52} & 0 & 0 & 0\\
        0 & 0 & \co(A)_3-a_{53}-a_{43} & 0 & 0\\
        0 & 0 & a_{52} & \co(A)_4-a_{54}+a_{43} & 0\\
        0 & 0 & 0 & 0 & \co(A)_5+a_{54}+a_{53}
    \end{psmallmatrix}, \\
       &(B_{10}, \emptyset)=
    \begin{psmallmatrix}
        \co(A)_1 & 0 & 0 & 0 & 0\\
        0 & \co(A)_2-a_{52} & 0 & 0 & 0\\
        0 & a_{52} & \co(A)_3-a_{53}-a_{43} & 0 & 0\\
        0 & 0 & 0 & \co(A)_4-a_{54}+a_{43} & 0\\
        0 & 0 & 0 & 0& \co(A)_5+a_{54}+a_{53}
    \end{psmallmatrix}, \\
    &(B_{11}, \emptyset)=
    \begin{psmallmatrix}
        \co(A)_1 & 0 & 0 & 0 & 0\\
        0 & \co(A)_2 & 0 & 0 & 0\\
        0 & 0 & \co(A)_3-a_{53}-a_{43} & 0 & 0\\
        0 & 0 & a_{43} & \co(A)_4-a_{54} & 0\\
        0 & 0 & 0 & 0 & \co(A)_5+a_{54}+a_{53}
    \end{psmallmatrix}, \\
    &(B_{12}, \emptyset)=
    \begin{psmallmatrix}
        \co(A)_1 & 0 & 0 & 0 & 0\\
        0 & \co(A)_2 & 0 & 0 & 0\\
        0 & 0 & \co(A)_3-a_{53} & 0 & 0\\
        0 & 0 & 0 & \co(A)_4-a_{54} & 0\\
        0 & 0 & 0 & a_{53} & \co(A)_5+a_{54}
    \end{psmallmatrix}, \\
    &(B_{13}, \emptyset)=
    \begin{psmallmatrix}
        \co(A)_1 & 0 & 0 & 0 & 0\\
        0 & \co(A)_2 & 0 & 0 & 0\\
        0 & 0 & \co(A)_3-a_{53} & 0 & 0\\
        0 & 0 & a_{53} & \co(A)_4-a_{54} & 0\\
        0 & 0 & 0 & 0 & \co(A)_5+a_{54}
    \end{psmallmatrix}, \\
    &(B_{14}, \emptyset)=
    \begin{psmallmatrix}
        \co(A)_1 & 0 & 0 & 0 & 0\\
        0 & \co(A)_2 & 0 & 0 & 0\\
        0 & 0 & \co(A)_3 & 0 & 0\\
        0 & 0 & 0 & \co(A)_4-a_{54} & 0\\
        0 & 0 & 0 & a_{54} & \co(A)_5
    \end{psmallmatrix}.
\end{flalign*}
  Then $[A]_{\Delta_{+}}^{\leq}\ast [A]_{\Delta_{-}}^{>}=[B_{1}]_{\emptyset}\ast \cdots [B_{14}]_{\emptyset}$ and the leading term is 
  \begin{flalign*}
    \begin{psmallmatrix}
    \co(A)_1 & a_{12} & a_{13} & a_{14} & 0\\
    0 &\co(A)_2-a_{12}-a_{52} & a_{23} & 0& 0\\
    0 & 0 & a_{33}& 0 & 0\\
    0 & 0 & a_{43}& a_{24}+a_{34}+a_{44}& 0\\
     0 & a_{52} & a_{53}& a_{54}& \co(A)_5\\
  \end{psmallmatrix}.
   \end{flalign*}
  
  Then we start to fill the elements $a_{ij}$ where $(i, j)\in \Delta$, we first fill the element $a_{i_1, j_1}=a_{24}$. 
  Take $\co(D_{2, 3,4}+a_{24}E_{34})=\ro(B_{1})$ and multiply with $[D_{2, 3,4}+a_{24}E_{34}]_{\emptyset}$, we have the leading term as: 
  \begin{flalign*}
    \begin{psmallmatrix}
    \co(A)_1 & a_{12} & a_{13} & a_{14} & 0\\
    0 &\co(A)_2-a_{12}-a_{52} & a_{23} & 0& 0\\
    0 & 0 & a_{33}& a_{24} & 0\\
    0 & 0 & a_{43}& a_{34}+a_{44}& 0\\
     0 & a_{52} & a_{53}& a_{54}& \co(A)_5\\
  \end{psmallmatrix}.
  \end{flalign*}
  Now we are ready to obtain the $\Delta$, we first multiply by $L_{D, 3}$, where $\co(D)=\ro(D_{2, 3,4}+a_{24}E_{34})$, 
   due to Lemma \ref{leading} $(b)$, we have the leading term as follows with coefficient $1$:
    \begin{flalign*}
    \begin{psmallmatrix}
    \co(A)_1 & a_{12} & a_{13} & a_{14} & 0\\
    0 &\co(A)_2-a_{12}-a_{52} & a_{23} & 0& 0\\
    0 & 0 & a_{33}& \senc{a_{24}} & 0\\
    0 & 0 & a_{43}& a_{34}+a_{44}& 0\\
     0 & a_{52} & a_{53}& a_{54}& \co(A)_5\\
  \end{psmallmatrix}.
  \end{flalign*}
  Then we multiply by $[D_{2,2,4}+a_{24}E_{23}]_{\emptyset}$ and use $(a)$ in Lemma \ref{leading}, we have
    \begin{flalign*}
    \begin{psmallmatrix}
    \co(A)_1 & a_{12} & a_{13} & a_{14} & 0\\
    0 &\co(A)_2-a_{12}-a_{52} & a_{23} & \senc{a_{24}}& 0\\
    0 & 0 & \senc{a_{33}}& 0 & 0\\
    0 & 0 & a_{43}& a_{34}+a_{44}& 0\\
     0 & a_{52} & a_{53}& a_{54}& \co(A)_5\\
  \end{psmallmatrix}.
  \end{flalign*}
  Thanks to $(c)$ in Lemma \ref{leading}, the remaining `circles' in $(A, \Delta)$ will naturally 
  appear when filling the corresponding elements in the order as we have stated.
\end{proof}

\subsection{Modified Mirabolic Quantum $\glln$ and its BLM Basis}

Let $\widetilde{\Xi}=\{(A, \Delta)\ |\ A=(a_{ij})\in Mat_{n\times n}(\mathbb{Z}), a_{i j}\geq 0 \text{ for } i\neq j\}$ 
where $\Delta=\{(i_1,j_1),\ldots,(i_k,j_k)\}$ satisfies 
$$
1\leq i_1<\cdots<i_k\leq n,\quad
1\leq j_k<\cdots<j_1\leq n,\\
a_{ij}>0\text{ whenever }(i,j)\in\Delta
\text{ and }i\neq j.
$$
Thus, unlike the definition of $\Xi_{n|n}$, no positivity condition is imposed on $a_{ii}$ when $(i,i)\in\Delta$.

For $(A,\Delta)\in\widetilde{\Xi}$ and $p\in\mathbb Z$, set $(_{p}A, \Delta)=(A+pI, \Delta)$ where $I$ is the identity matrix. 
When $p$ is large enough, we have $({}_pA, \Delta)\in \Xi_{n|n}$.

Let $v'$ be another indeterminate. 
The following proposition in \cite{FZM25} established the stabilization property of the mirabolic $q$-Schur algebra.

\begin{prop}\cite[Proposition 5.1.1]{FZM25}\label{stab}
  For $(A_1, \Delta_1), (A_2, \Delta_2), \cdots, (A_k, \Delta_k)\in \widetilde{\Xi}$ satisfying $\co(A_i)=\ro(A_{i+1})$ for $i=1, \cdots, k-1$.
  There exist $(Z_1, \Delta'_1), \cdots, (Z_m, \Delta'_m)$$\in \widetilde{\Xi}$, elements $G_{\Delta'_i}(v, v')\in \mbb Q(v)[v', v'^{-1}]$
  and an integer $p_{0}\geq 0$ such that for all $p \geq p_0$
  $$
  [_{p}A_{1}]_{\Delta_1} \ast [_{p}A_{2}]_{\Delta_2} \ast \cdots \ast [_{p}A_{k}]_{\Delta_k} =
  \sum_{i=1}^{m} G_{\Delta'_i}(v, v^{-p})[_{p}Z_{i}]_{\Delta'_i}.
  $$
\end{prop}

We assign each $\Delta$ a subset $I(\Delta)$ of $[1, n]\times [1, n]$, defined as
$$
I(\Delta)=\{(r,s)\mid \exists (i, j)\in \Delta \text{ s.t. } r\leq i, s\leq j\}.
$$
and recall in \cite{BLM90}, for any element in $\{A\in Mat_{n\times n}(\mbb Z)\mid a_{ij}\geq0 \text{ for } i\neq j\}$, one has
$$
\Psi(A)= \sum_{1\leq i<j\leq n}\frac{(j-i)(j-i+1)}{2} (a_{ij}+a_{ji}).
$$
We equip $\mathbb N^2$ with the lexicographic order and define $\Phi(A,\Delta)=(\Psi(A), \sharp I(\Delta))\in \mbb N^2$.
By definition, we have the following lemma.

\begin{lem}
    For any $(A',\Delta'), (A,\Delta)\in \widetilde{\Xi}$, we have the following results.
    \begin{enumerate}[$(a)$]
        \item For any integer $p$, $\Phi(_pA,\Delta)=\Phi(A,\Delta)$
        \item if $(A',\Delta')<(A,\Delta)$ then we have $\Phi(A',\Delta')<_{\mrm {lex}} \Phi(A,\Delta)$.
    \end{enumerate}
\end{lem}

\begin{prop}\label{bar-inv}
    For any $(A,\Delta)\in \widetilde{\Xi}$, there exist $(T_1, \Delta_1), (T_2, \Delta_2), \cdots, (T_m, \Delta_m)\in \widetilde{\Xi}$, 
    elements $H_i(v,v^{'})\in \mathbb{Q}(v)[v', v'^{-1}]$, $1\leq i\leq m$ and an integer $p_0\geq 0$ such that
    $$
    \overline{[_{p}A]_{\Delta}}=\sum_{i=1}^{m}H_{i}(v, v^{-p})[_{p}T_i]_{\Delta_i}
    $$
    for all $p\geq p_0$.
\end{prop}
\begin{rem}
    It is easy to see that $H_i(v,v^{-p})$ and $G_{\Delta'_i}(v, v^{-p})\in \mathcal A$.
\end{rem}
\begin{proof}
    We prove the statement by induction on $\Phi(A,\Delta)$. For $\Phi(A,\Delta)=(0, 0)$, we have $A$ is diagonal and $\Delta=\emptyset$, so $\overline{[_pA]_{\Delta}}=[_pA]_{\Delta}$.
    
    Now assume that $(0, 0)<_{\mrm{lex}} \Phi(A,\Delta)$ and the statement has been proved for every $(A', \Delta')\in \widetilde{\Xi}$ 
    satisfying $\Psi(A',\Delta')<_{\mrm{lex}}\Phi(A,\Delta)$. And we replace $A$ with $_rA$ if necessary so that we can assume $(A,\Delta)\in \Xi_{n|n,d}$.
    
    By Proposition \ref{mon-b} and Proposition \ref{stab}, we have, for all sufficient large integer $p$ 
    $$
    [_pA]_{\Delta}=[_pB_1]_{\emptyset}\ast \cdots \ast L_{_pD, t} \ast \cdots \ast [_pB_N]_{\emptyset}-\sum_{i=2}^{m} G_{\Delta'_i}(v, v^{-p})[_{p}Z'_{i}]_{\Delta^{'}_i},
    $$
    where $(_pZ'_i, \Delta^{'}_i)\in \widetilde{\Xi}$ satisfy $\Psi(_pZ_i, \Delta^{'}_i)<_{\mrm{lex}}\Psi(_pA, \Delta)$ and $G_{\Delta'_i}(v, v^{-p})\in \mathfrak{R}$.
    
   
    Now we apply the bar-involution to both sides. Due to the induction hypothesis, we only need to check that 
    $[_pB_1]_{\emptyset}\ast \cdots \ast L_{_pD, t} \ast \cdots \ast [_pB_N]_{\emptyset}$ is invariant under the involution. 
    We can regard the elements $[{}_pB_i]_{\emptyset}$ as elements in $q$-Schur algebra, and we have seen that 
    the bar involution is the module homomorphism on the mirabolic $q$-Schur algebra as a bimodule of $q$-Schur algebra, 
    so we only need to prove
    $$
    \overline{[_pB_i]_{\emptyset}}=[_pB_i]_{\emptyset}\quad \overline{L_{_pD, t}}=L_{_pD, t}.
    $$
    The first is true because $(_pB_i, \emptyset)$ corresponds to the closed orbit and identity \eqref{bar}.
    The second one is true by the construction of canonical basis of the mirabolic $q$-Schur algebra and $L_{_pD, t}=\{{}_{p}D\}_{\{(t, t)\}}$.
\end{proof}

    Let $v'=1$, we obtain the stabilization algebra $\mathcal K$ of \cite{FZM25}. We refer this algebra as the $\mathcal{A}$-form of modified mirabolic quantum $\mathfrak{gl}_n$ and denote it by $_{\mcal{A}}\mmuvgln$ and denote by $\mmuvgln$ the algebra $\mbb{Q}(v)\otimes_{\mcal{A}}\mmuvgln$. 
    Then elements $[A]_{\Delta}$ for $(A,\Delta)\in \widetilde{\Xi}$ form a basis of  $_{\mcal{A}}\mmuvgln$ and therefore a basis of $\mmuvgln$.
    
    The multiplication is defined as 
    $$
    [A]_{\Delta}\cdot [A']_{\Delta'}=
    \begin{cases}
        \sum_{i=1}^{m} G_{\Delta'_i}(v, 1)[Z_{i}]_{\Delta'_i} & \text{ if } \co(A)=\ro(A') \\
        0 &\text { otherwise. }
    \end{cases}
    $$
    Furthermore, Proposition \ref{bar-inv} and \cite[Proposition 4.3]{BLM90} show that there is a $\dot{U}_{v}(\glln)$-bimodule homomorphism $\bar{}:\mmuvgln\to \mmuvgln$ defined by
    $$
    \overline{v^{j}[A]_{\Delta}}=v^{-j}\sum_{i=1}^{m}H_{i}(v, 1)[T_i]_{\Delta_i},
    $$
    whose square is the identity.

    For any $(A,\Delta)\in \widetilde{\Xi}$, where $\Delta=\{(i_1, j_1), \cdots, (i_k, j_k)\}$, 
    we can define similar monomials $[A]_{\Delta_{-}}^{\leq}$, $[A]_{\Delta_{-}}^{>}$, $[A]_{\Delta_{+}}^{\leq}$, $[A]_{\Delta_{+}}^{>}$ in $\mmuvgln$ 
    as Proposition \ref{mon-b}. 
    For $(D, \{(h, h)\})\in \widetilde{\Xi}$, we also define 
    $$L_{D, h}=\sum_{i=1}^{h}v^{-\sum_{k=i+1}^{h}d_k}[D]_{\{(i, i)\}}+ v^{-\sum_{1\leq i\leq h}d_i}[D]_{\emptyset}.$$ 
    
    We can also define a similar partial order for $\widetilde{\Xi}$, that is, for $(A, \Delta), (A',\Delta')\in \widetilde{\Xi}$, 
    we say $(A',\Delta')\sqsubseteq (A,\Delta)$ if
    \begin{itemize}
        \item $A'\sqsubset A$
        \item or $A'=A$ and $\Delta'\leq \Delta$.
    \end{itemize}
    Then if $(A',\Delta')\sqsubseteq (A,\Delta)$ and $(A',\Delta')\neq (A,\Delta)$, we say
    $(A',\Delta')\sqsubset (A,\Delta)$.
    
    According to the multiplication formulas of $\mmuvgln$ (cf. \cite[Corollary 5.2.1-5.2.2]{FZM25}), 
    for any $(A,\Delta)\in \widetilde{\Xi}$, where $\Delta=\{(i_1, j_1), \cdots, (i_k, j_k)\}$, take $\vartheta(\Delta)$ as $\eqref{circ start position}$, 
    Proposition \ref{mon-b} implies the following identities.
    \begin{enumerate}[ $(a)$ ]
        \item Suppose $i_1\leq j_1$.
          \begin{align}
        &[A]_{\Delta_{-}}^{\leq}\cdot [A]_{\Delta_{+}}^{>}\cdot \prod^{\longrightarrow}_{h\in [i_1, \vartheta(\Delta)-1]}[D_{i_1, h, j_1}+a_{i_1,j_1}E_{h, h+1}]_{\emptyset}\notag
        \\
        &\cdot L_{D, \vartheta(\Delta)}\cdot \prod^{\longrightarrow}_{h\in [\vartheta(\Delta), j_1-1]}[D_{i_1, h, j_1}+a_{i_1,j_1}E_{h, h+1}]_{\emptyset}\ast [A]_{\Delta_{+}}^{\leq}\ast [A]_{\Delta_{-}}^{>}\notag\\
        =& [A]_{\Delta}+\text{ lower terms}\label{mmon-u},
            \end{align}
        \item If $i_1>j_1$
         \begin{align}
        &\prod^{\longrightarrow}_{1\leq i\leq h<j\leq n}[D_{i, h, j}+a_{ij}E_{h, h+1}]_{\emptyset}\cdot [A]_{\Delta_{-}}^{\leq}\cdot L_{D, j_1}\cdot [A]_{\Delta_{-}}^{>}\notag\\
        &= [A]_{\Delta}+\text{ lower terms}\label{mmon-l},
        \end{align}
    \end{enumerate}
    where the product is taken in $\mmuvgln$, $(D, \{(\vartheta(\Delta), \vartheta(\Delta))\})$ and $(D, \{(j_1, j_1)\})$ are suitable elements in $\widetilde{\Xi}$, the diagonal matrices $D_{i, h, j}$ are uniquely determined.
    The elements on the left-hand side of identities (\ref{mmon-u}) and (\ref{mmon-l}) will be denoted by $M_{A, \Delta}$, they form the monomial basis of $_{\mcal A}\mmuvgln$.

  By the definition, this basis is fixed under the bar involution, and therefore we have
  $$
  \overline{[A]_{\Delta}}=\sum_{(A',\Delta')\sqsubseteq (A,\Delta)}\tau_{(A, \Delta), (A',\Delta')}[A']_{\Delta'}.
  $$
 where $\tau_{(A, \Delta), (A',\Delta')}\in \mathcal{A}$ and $\tau_{(A, \Delta), (A,\Delta)}=1$.

  By the standard argument in \cite[7.10]{L90}, we obtain the following result.
  \begin{prop}\label{bas-gln}
     For every $(A,\Delta)\in\widetilde{\Xi}$, there exists a unique element $\{A\}_{\Delta}\in {}_{\mathcal{A}}\mmuvgln$ such that $\overline{\{A\}_{\Delta}}=\{A\}_{\Delta}$
      and 
      $$\{A\}_{\Delta}=[A]_{\Delta}+\sum_{(A', \Delta')\sqsubset (A,\Delta)} \pi_{(A,\Delta), (A', \Delta')}[A']_{\Delta'},$$
      where $\pi_{(A,\Delta), (A', \Delta')}\in v^{-1}\mathbb{Z}[v^{-1}]$ for $(A',\Delta')\sqsubset(A, \Delta)$.
      The elements $\{A\}_{\Delta}$ form an $\mcal{A}$-basis of $_{\mcal{A}}\mmuvgln$ and therefore a basis of $\mmuvgln$.
  \end{prop}

  Recall the surjective algebra homomorphism $\Phi_d:\mmuvgln\to {}_{\mbb Q}\MS$
  $$
  \Phi_d([A]_{\Delta})=\begin{cases}
    [A]_{\Delta}  & \text{ if } (A,\Delta)\in \Xi_{n|n, d},\\
     0 & \text{ otherwise }.
  \end{cases}
  $$
  We have the following lemma.
  \begin{lem} \label{cb-gln to schur}
      For any $(A,\Delta)\in \widetilde{\Xi}$, we have
       $$
 \Phi_d(\{A\}_{\Delta})=\begin{cases}
    \{A\}_{\Delta}  & \text{ if } (A,\Delta)\in \Xi_{n|n, d},\\
     0 & \text{ otherwise }.
  \end{cases}
  $$
  \end{lem}
\begin{proof}
     We first show that the $\Sd$-module homomorphism $\bar{}$ is commutes with $ \Phi_d$. 
     
     It can be seen directly from definition that 
     $$
      \Phi_d(M_{A,\Delta})=m_{A, \Delta}.
     $$
     Since both $M_{A,\Delta}$ and $m_{A,\Delta}$ are bar-invariant, it follows that
     \begin{align*}
         \Phi_d(\overline{M_{A,\Delta}})&= \Phi_d(M_{A,\Delta})\\
        &=m_{A,\Delta}=\overline{m_{A,\Delta}}\\
        &=\overline{ \Phi_d(M_{A,\Delta})}.
     \end{align*}
    Therefore, we have
    $$
     \overline{ \Phi_d(\{A\}_{\Delta})}= \Phi_d(\overline{\{A\}_{\Delta}})= \Phi_d(\{A\}_{\Delta}).
    $$
    
    By Proposition \ref{bas-gln}, we have 
    $$\{A\}_{\Delta}=[A]_{\Delta}+\sum_{(A', \Delta')\sqsubset (A,\Delta)} \pi_{(A,\Delta), (A', \Delta')}[A']_{\Delta'}.$$
    Applying $ \Phi_d$ to both sides gives
     $$ \Phi_d(\{A\}_{\Delta})= \Phi_d([A]_{\Delta})+\sum_{\substack{(A',\Delta')\in \widetilde{\Xi}\\(A', \Delta')\sqsubset (A,\Delta)}} \Phi_d(\pi_{(A,\Delta), (A', \Delta')}[A']_{\Delta'}).$$
     If $(A,\Delta)\notin \Xi_{n|n, d}$, then we have 
     \begin{align*}
      \Phi_d(\{A\}_{\Delta})&=\sum_{\substack{(A',\Delta')\in \widetilde{\Xi}\\(A', \Delta')\sqsubset (A,\Delta)}}\pi_{(A,\Delta), (A', \Delta')} \Phi_d([A']_{\Delta'})\\
     &=\sum_{\substack{(A',\Delta')\in \Xi_{n|n,d}\\(A', \Delta')< (A,\Delta)}}\pi_{(A,\Delta), (A', \Delta')}[A']_{\Delta'},
     \end{align*}
    where $\pi_{(A,\Delta), (A', \Delta')}\in v^{-1}\mbb Z[v^{-1}]$.
    Then $ \overline{\Phi_d(\{A\}_{\Delta})}= \Phi_d(\{A\}_{\Delta})$ shows that we must have $\Phi_d(\{A\}_{\Delta})=0$.
    
    If $(A,\Delta)\in \Xi_{n|n, d}$, 
     \begin{align*}
      \Phi_d(\{A\}_{\Delta})- \Phi_d([A]_{\Delta})&=\sum_{\substack{(A',\Delta')\in \widetilde{\Xi}\\(A', \Delta')\sqsubset (A,\Delta)}}\pi_{(A,\Delta), (A', \Delta')} \Phi_d([A']_{\Delta'})\\
     &=\sum_{\substack{(A',\Delta')\in \Xi_{n|n,d}\\(A', \Delta')< (A,\Delta)}}\pi_{(A,\Delta), (A', \Delta')}[A']_{\Delta'},
     \end{align*}
    where $\pi_{(A,\Delta), (A', \Delta')}\in v^{-1}\mathbb{Z}[v^{-1}]$ by definition.
    Therefore $\{ \Phi_d(\{A\}_{\Delta})\mid (A,\Delta)\in \Xi_{n|n,d}\}$ forms a basis of $\MS$ and satisfy
    $\overline{\Phi_d(\{A\}_{\Delta})}=\Phi_d(\{A\}_{\Delta})$ and
    $$
    \Phi_d(\{A\}_{\Delta})\in [A]_{\Delta}+\sum_{\substack{(A',\Delta')\in \Xi_{n|n,d}\\(A', \Delta')\sqsubset (A,\Delta)}}v^{-1}\mathbb{Z}[v^{-1}][A']_{\Delta'}.
    $$
    By the uniqueness of the canonical basis of $\MS$, we have $\Phi_d(\{A\}_{\Delta})=\{A\}_{\Delta}$ as desired.
    \end{proof}

    For any $\lambda\in \mathbb{Z}^n$, we define the element $(D_{\lambda},\emptyset)\in \widetilde{\Xi}$ where $D_{\lambda}=\mathrm{diag}(\lambda_1,\cdots, \lambda_n)$ and $1_{\lambda}=[D_{\lambda}]_{\emptyset}$. As we have seen in \cite{FZM25}, the algebra $\mmuvgln$ has the basis $\{[A]_{\Delta}\ |\ (A,\Delta)\in \widetilde{\Xi}\}$. Thus
    we have 
    $$
    1_{\lambda}[A]_{\Delta}1_{\mu}=\delta_{\ro(A), \lambda}\delta_{\co(A),\mu}[A]_{\Delta}
    $$
    Consequently,
    $$
    \mmuvgln=\bigoplus_{\lambda, \mu \in \mathbb Z^n} 1_{\lambda}\mmuvgln1_{\mu}.
    $$
    For $\lambda\in\mathbb Z^n$, set $|\lambda|=\sum_i \lambda_i$. 
    For $d\in\mathbb Z$, define 
    $$\mmuvgln\langle d\rangle=\bigoplus_{\substack{\lambda,\mu\in \mbb Z^n\\ |\lambda|=|\mu|=d}} 1_{\lambda}\mmuvgln1_{\mu}.$$
    We then obtain the block decomposition
    $$
    \mmuvgln=\bigoplus_{d\in \mathbb Z} \mmuvgln\langle d\rangle.
    $$

    Recall the $\mathbb{Q}(v)$-algebra $\widehat{\mathcal{K}}$ consists 
	of all formal $\mathbb{Q}(v)$-linear combinations $\sum_{(A, \Delta)\in \widetilde{\Xi} }\xi_{(A, \Delta)}[A]_{\Delta}$, such that, 
    for every $a\in \mathbb{Z}^n$, 
	the two set $\{(A, \Delta)\in\tilde{\Xi} \mid\xi _{(A, \Delta)}\neq 0 \ , \ro(A)=a\}$ 
	and $\{(A, \Delta)\in\tilde{\Xi}\mid \xi _{(A, \Delta)}\neq 0 \ , \co(A)=a\}$ are finite. 
	
	Let $\Xi^0$ be the subset consisting of all $(A, \Delta)\in \widetilde{\Xi}$ such that all diagonal entries of $A$ are zero. 
	For $(A,\Delta)\in \Xi^{0}$ and $\mathbf{j}=(j_1,\cdots, j_n)\in \mathbb{Z}^n$, we can define elements $A(\mathbf{j})_{\Delta}$ in $\widehat{\mathcal{K}}$ 
	$$
	A(\mathbf{j})_{\Delta}=\sum_{Z} v^{z_1j_1+ \cdots +z_n j_n}[A+Z]_{\Delta},
	$$
	where $Z$ runs over all diagonal matrices $diag(z_1, \cdots, z_n)$ in $Mat_{n\times n}(\mathbb{Z})$.

    For any $1\leq a\leq n$, let $\mathbf{j}_a=(0,\cdots, 1,\cdots ,0)$ (on the $a$-th position).
    Then the mirabolic quantum $\muvgln$ is the subalgbra of $\widehat{\mcal K}$ generated by $L=0(-2\mathbf{j}_1)_{\emptyset}+0(-\mathbf{j}_1)_{\{(1, 1)\}}$, $H_a=0(\mathbf{j}_a)_{\emptyset}$, $E_i=E_{i, i+1}(\mathbf{0})_{\emptyset}$ and $F_i=E_{i+1, i}(\mathbf{0})_{\emptyset}$ for $1\leq i\leq n-1$ and $1\leq a\leq n$. 

    Then we have the algebra $\mmuvgln$ is generated by $E_i1_{\lambda}$, $F_i1_{\lambda}$, $L1_{\lambda}$ and $1_{\lambda}$ where $\lambda\in \mathbb{Z}^n$\cite{FZM25}. 


\section{The Canonical Basis of $\mmuvsln$}
\subsection{Comodule structure on $\mcal{MS}_{n, d}$}

Now we consider three positive integers $d, d', d''$ such that $d=d'+d''$, and fix an isomorphism $\mbb F_q^{d}=\mbb F_{q}^{d'}\oplus \mbb F_{q}^{d''}$. Let $V=F_{q}^{d}$, $V'=\Fqdp$ and $V''=\Fqdpp$. Take $\pi'$ as the projection $\Fqd\to \Fqdp$ and $\pi''$ as the operation of taking the intersection with $\mbb F_{q}^{d''}$. Then for any flag $f=0\subseteq V_1\subseteq V_2\subseteq\cdots\subseteq V_n=\Fqd$ in $\mscr X_{n, d}$ induces two flags:
\begin{align*}
    &\pi'(f)=0\subseteq (V_1+\Fqdpp)/\Fqdpp \subseteq \cdots \subseteq(V_n+\Fqdpp)/\Fqdpp=\Fqdp \in \mscr X_{n,d'},\\
    &\pi''(f)=0\subseteq V_1\cap \Fqdpp\subseteq\cdots  \subseteq V_n\cap \Fqdpp=\Fqdpp \in \mscr X_{n,d''}.
\end{align*}
We can assign each pair $(f',f'')\in \mscr X_{n, d'}\times \mscr X_{n, d''}$ a set
$$
Z_{f',f''}=\{f\in \mscr X_{n, d} | \pi'(f)=f', \pi''(f)=f''\}.
$$
Moreover, for each triple $(f', f'', w')\in \Xp\times \Xpp\times \Fqdp$, we define
$$
Z_{f',f'', w'}=\{(f,w)\in \mscr X_{n, d}\times \Fqd | \pi'(f)=f', \pi''(f)=f'', \pi'(w)=w'\}.
$$

We define a linear map
$$
\tilde{\rho}: \mcal{MS}_{n,d}\to \mcal{S}_{n,d''}\otimes \mcal{MS}_{n,d'}
$$
by $\tilde{\rho}(h)(\hat{f}', f', \hat{f''}, f'',w')=\sum_{(f, w)\in Z_{f', f'', w'}}h(\hat{f},f,w)$, for any $(\hat{f}', f', \hat{f''}, f'',w')\in \Xp\times \Xp\times \Xpp\times \Xpp\times \Fqdp$,
where $\hat{f}\in Z_{\hat{f}', \hat{f}''}$ is a fixed element. By a similar argument in \cite{L00}, we can see that map $\tilde{\rho}$ is independent of the choice of $\hat{f}$.

Recall the algebra homomorphism in \cite{FL21}: $\tilde{\Delta}: \Sd\to \mcal S_{n,d''}\otimes \mcal S_{n,d'}$, defined by $\tilde{\Delta}(h)(\hat{f}'', f'', \hat{f'}, f')=\sum_{f\in Z_{f', f''}}h(\hat{f},f)$.
\begin{prop}\label{cm-hom}
    The map $\tilde{\rho}$ is a homomorphism of $\Sd$-modules, i.e., for any $h_1\in \Sd, h_2\in \mcal{MS}_{n,d}$, we have
    $$\tilde{\rho}(h_1\ast h_2)=\tilde{\Delta}(h_1)\ast \tilde{\rho}(h_2).$$
\end{prop}
\begin{proof}
Fix $\widehat f\in Z_{\widehat f',\widehat f''}$. By definition, 
    \begin{align*}
        \tilde{\rho}(h_1\ast h_2)(\hat{f}', f', \hat{f''}, f'',w')&=\sum_{(f,w)\in Z_{f', f'', w'}}h_1\ast h_2(\hat{f},f,w)\\
        & =\sum_{(f,w)\in Z_{f', f'', w'}}\sum_{\tilde{f}\in \mscr X_{n,d}}h_1(\hat{f}, \tilde{f})h_2(\tilde{f}, f, w).
    \end{align*}
On the other hand
\begin{align*}
    &\tilde{\Delta}(h_1)\ast \tilde{\rho}(h_2)(\hat{f}'', f'', \hat{f'}, f',w')\\
    &=\sum_{\substack{\widetilde f'\in\mathscr X_{n,d'}\\ \widetilde f''\in\mathscr X_{n,d''}}}\tilde{\Delta}(h_1)(\hat{f}'', \tilde{f}'', \hat{f}', \tilde{f}')\tilde{\rho}(h_2)(\tilde{f}'',f'', \tilde{f}',f', w')\\
    &= \sum_{\substack{\widetilde f'\in\mathscr X_{n,d'}\\ \widetilde f''\in\mathscr X_{n,d''}}}
    \sum_{\widetilde f\in Z_{\widetilde f',\widetilde f''}}h_1(\widehat f,\widetilde f)
     \sum_{(f,  w)\in Z_{f',f'', w'}}h_2(\widetilde f, f, w)\\
    &=\sum_{\tilde{f}\in \mscr{X}_{n, d}}\sum_{(f, w)\in Z_{f', f'', w'}}h_1(\hat{f}, \tilde{f})h_2(\tilde{f}, f, w).\\
\end{align*}
This proves the assertion.
\end{proof}

Let $L\in \mathcal{MS}_{n,d}$ be
$$
L(f, f', w)=\begin{cases}
    v^{-2\dim V_1} & \text{ if } f=f', w\in V_1;\\
    0 & \text{ otherwise. }
\end{cases}
$$

\begin{prop}\label{cm-f}
    For any $i\in [1, n-1]$, we have
    \begin{align*}
    \tilde{\rho}(K_{i}^{\pm 1})&=\tilde{\Delta}(K_{i}^{\pm 1})=K_{i}''^{\pm 1}\otimes K_{i}'^{\pm 1};\\
     \tilde{\rho}(E_{i})&=\tilde{\Delta}(E_{i})=H''^{-1}_{i}\otimes E'_{i} +E''_i\otimes H'_{i};\\
      \tilde{\rho}(F_{i})&=\tilde{\Delta}(F_{i})=H''_{i+1}\otimes F'_{i} +F''_i\otimes H'^{-1}_{i+1};\\
       \tilde{\rho}(L)&=1\otimes L'.
\end{align*}
\end{prop}
\begin{proof}
  The first three identities follow from definition of $\tilde{\rho}$ and \cite[Lemma 1.3.2]{FL21}. We prove the last identity. 
  
    By definition, 
    \begin{align*}
        \tilde{\rho}(L)(\hat{f}', f', \hat{f}'', f'',w')=\sum_{(f, w)\in Z_{f', f'', w'}}L(\hat{f}, f, w)=v^{-2\dim V_1}N,
    \end{align*}
    where $N=\sharp\{(f,w)\in Z_{f', f'', w'}|f=\hat{f}, w\in V_1\}$.
    
    Since $\hat{f}$ is a fixed element in $\mscr X_{n,d}$, the set $\{(f,w)\in Z_{f', f'', w'}|f=\hat{f}, w\in V_1\}=\{w\in V_1 | \pi'(w)=w'\}$. 
    Therefore
    $N=q^{\dim \pi''(V_1)}$.
    So 
\begin{align*}
    \tilde{\rho}(L)(\hat{f}'', f'', \hat{f}', f',w')&=v^{-2\dim V_1+2\dim \pi''(V_1)}\\
    &=v^{-2\dim \pi'(V_1)}=(1 \otimes L')(\hat{f}'', f'', \hat{f}', f',w').
\end{align*}

\end{proof}

Let
$$
\Lambda_{n,d}=\{\mbf a=(a_1,\cdots, a_n)\in \mbb N^{n}|\sum_{i}a_{i}=d\}.
$$
For $\mathbf a\in\Lambda_{n,d}$, set
$$
\X(\mbf a)=\{f\in \X | \dim V_i/V_{i-1}=a_i, \forall i\in [1, n]\}.
$$
Then
$$\mscr X_{n, d}=\bigsqcup_{\mbf a\in \Lambda_{n, d}}\mscr{X}_{n, d}(\mbf a).$$

Let $\mcal{MS}_{n,d}(\mbf b,\mbf a)$ be the subspace of $\MS$ consisting of functions supported on $\X(\mbf b)\times \X(\mbf a)\times V$ 
and let $\Sd(\mbf b,\mbf a)$ be the subspace of $\Sd$ consisting of all functions supported on $\X(\mbf b)\times \X(\mbf a)$.

For $\fb=\fb'+\fb''$ and $\fa=\fa'+\fa''$, the linear map $\tilde{\rho}$ induces a linear map:
$$
\tilde{\rho}_{\mbf b',\mbf a',\mbf b'',\mbf a''}:\MS(\mbf b,\mbf a)\to \mathcal{S}_{n, d''}(\mbf b'', \mbf a'')\otimes \MSp(\mbf b',\mbf a'),
$$
by restricting $\tilde{\rho}$ to the space $\MS(\mbf b,\mbf a)$ and projecting down to $\mathcal{S}_{n, d''}(\mbf b'', \mbf a'')\otimes \MSp(\mbf b',\mbf a')$. 
Moreover, 
$$
\tilde{\rho}=\oplus\tilde{\rho}_{\mbf b',\mbf a',\mbf b'',\mbf a''},
$$
where the sum runs over all $\fb, \fa\in \Lambda_{n,d}$, $\fb',\fa'\in \Lambda_{n,d'}$, $\fb'',\fa''\in \Lambda_{n,d''}$ such that $\fb=\fb'+\fb''$ and $\fa=\fa'+\fa''$.
Take $\rho_{\fb',\fa',\fb'',\fa''}=v^{\sum_{1\leq i\leq j\leq n}b_i'b_j''-a_i'a''_{j}}\tilde{\rho}_{\fb',\fa',\fb'',\fa''}$ and $\rho=\oplus \rho_{\fb',\fa',\fb'',\fa''}$.
Recall the algebra homomorphism $\Delta_{\mbf v}$ defined in \cite{FL21} 
\begin{prop}
    The linear map $\rho$ is a homomorphism of $\Sd$-module. More precisely,
  $$
  \rho(h_1\ast h_2)
 =
\Delta_v(h_1)\ast\rho(h_2),
  $$
  for any $h_1\in \Sd$ and $h_2\in \MS$.
    Furthermore, for every $1\leq i<n$, we have
    \begin{align*}
    \rho(K_{i}^{\pm 1})&=\Delta_{\mbf v}(K_{i}^{\pm 1})=K_{i}''^{\pm 1}\otimes K_{i}'^{\pm 1};\\
     \rho(E_{i})&=\Delta_{\mbf v}(E_{i})=1\otimes E'_{i} + E_i''\otimes K_i' ;\\
      \rho(F_{i})&=\Delta_{\mbf v}(F_{i})= K_i''{}^{-1}\otimes F'_{i}+  F''_i\otimes 1;\\
       \rho(L)&=1\otimes L'.
\end{align*}
\end{prop}
\begin{proof}
The compatibility
\[
\rho(h_1\ast h_2)
 =
\Delta_v(h_1)\ast\rho(h_2)
\]
follows from Proposition~\ref{cm-hom} and the additivity of the exponent.
    
    The formulas for $E_i,F_i$, and $K_i^{\pm1}$ follow from \cite[Proposition 1.3.3]{FL21}.
    Finally, every nonzero component of
    $\widetilde\rho(L)$ satisfies
    $\mathbf b'=\mathbf a'$ and $\mathbf b''=\mathbf a''$. Hence the exponent $\sum_{1\leq i\leq j\leq n}b_i'b_j''-a_i'a''_{j}$ is zero, and
    Proposition~\ref{cm-f} gives
    $$
    \rho(L)=1\otimes L'.$$
\end{proof}

\subsection{Inner product and transfer map of $\MS$}
    In this subsection, we aim to define the inner product and transfer map of $\MS$ for later use.
    For a flag
\[
f=(0=V_0\subseteq V_1\subseteq\cdots\subseteq V_n=V)\in \X
\]
set
$|f|_i=\dim(V_i/V_{i-1})$ for any $1\leq i\leq n$.

\begin{Def}\label{in-product}
    Define the bilinear form on $\MS$
    $$(\hspace{0.7em},\hspace{0.7em})_{d}:\MS\times \MS\to \mathbb{Q}(v)$$
    as follows. For $h, h'\in \MS$
    $$
    (h, h')_d=\sum_{f, f',\mu}v^{\sum_{i}|f|_i^2-|f'|_i^2}h(f, f', \mu)h'(f, f', \mu)
    $$
    where $(f, \mu)$ runs over $\X\times V$ and $f'$ runs over the set of representatives for the $G$-orbits on $\X$.
\end{Def}
Take $(A,\Delta)\in \Xi_{n|n, d}$, where $\Delta=\{(i_1, j_1),\cdots, (i_{k}, j_{k})\}$. 
Recall we have element $(A^t, \Delta^t)\in \Xi_{n|n, d}$, where $A^t$ is the transpose of $A$ and $\Delta^t=\{(j_k, i_k),\cdots, (j_{1}, i_{1})\}$.

Let $d_{A,\Delta}=d(A,\Delta)-r(A,\Delta)$, then
$$
2(d_{A,\Delta}-d_{A^t,\Delta^t})= \sum_{i}\ro(A)_i^2-\co(A)_i^2.
$$

We calculate the result of $(\hspace{0.7em},\hspace{0.7em})_{d}$ on the basis $\{e_{A,\Delta}\mid (A,\Delta)\in \Xi_{n|n, d}\}$.
For any $(A,\Delta)$ and $(A',\Delta')\in \Xi_{n|n,d}$, we have 
\begin{align*}
    (e_{A,\Delta}, e_{A', \Delta'})_d&=\sum_{f, f',\mu}v^{\sum_{i}|f|_i^2-|f'|_i^2}e_{A,\Delta}(f, f', \mu)e_{A',\Delta'}(f, f', \mu)\\
    &=\delta_{(A,\Delta), (A',\Delta')}\sum_{f, f',\mu}v^{\sum_{i}\ro(A)_i^2-\co(A)_i^2}e_{A,\Delta}(f, f', \mu)e_{A',\Delta'}(f, f', \mu)\\
    &=\delta_{(A,\Delta), (A',\Delta')}\sum_{f, f',\mu}v^{2d_{A,\Delta}-2d_{A^t,\Delta^t}}e_{A,\Delta}(f, f', \mu)e_{A',\Delta'}(f, f', \mu)\\
    &=\delta_{(A,\Delta), (A',\Delta')}q^{d_{A,\Delta}-d_{A^t,\Delta^t}}\sharp{X^{f'}_{A^t,\Delta^t}},
\end{align*}
where $X^{f'}_{A^t,\Delta^t}=\{(f, \mu)\in \X\times V|(f',f,\mu)\in \mcal O_{A^t,\Delta^t}\}$ and $f'$ is any fixed flag in $\X$ such that $\dim V'_i/V'_{i-1}=\co(A)_i$.

We further have
\begin{align*}
    ([A]_\Delta, [A']_{\Delta'})_d&=\delta_{(A,\Delta), (A',\Delta')}v^{-2d_{A,\Delta}}(e_{A,\Delta}, e_{A', \Delta'})_d\\
    &=\delta_{(A,\Delta), (A',\Delta')}v^{-2d_{A,\Delta}}v^{2d_{A,\Delta}-2d_{A^t,\Delta^t}}\sharp{X^{f'}_{A^t,\Delta^t}}\\
    &=\delta_{(A,\Delta), (A',\Delta')}q^{-d_{A^t,\Delta^t}}\sharp{X^{f'}_{A^t,\Delta^t}}.
\end{align*}
 
Let $G_{f'}$ be the stabilizer of $f'$ under $G$-action and similar we define $G_{f, \mu}$.
The stabilizer $G_{f'}$ acts transitively on $X^{f'}_{A^t, \Delta^t}$, and the stabilizer of the point
$(f,\mu)\in X^{f'}_{A^t, \Delta^t}$ under the $G_{f'}$-action is $G_{f'}\cap G_{f,\mu}=G_{f',f,\mu}$. 
Therefore, we have $X^{f'}_{A^t, \Delta^t}\cong G_{f'}/G_{f',f,\mu}$ and $X^{f'}_{A^t, \Delta^t}$ is irreducible.
Dimension of $X^{f'}_{A^t, \Delta^t}$ is:
\begin{align*}
    \dim X^{f'}_{A^t, \Delta^t}&=\dim G_{f'}-\dim G_{f',f,\mu}\\
    &= \dim G-\dim G\cdot f'-\dim G_{f',f,\mu}\\
    &= \dim \mcal{O}_{A^t,\Delta^t}-\dim G\cdot f'\\
    &=d_{A^t,\Delta^t}.
\end{align*}
Then we have the following lemma.
\begin{lem}\label{orth-f-stan}
  We have 
  \begin{enumerate}[$(a)$]
     \item $([A]_{\Delta}, [A]_{\Delta})_d\in 1+v^{-1}\mathbb{Z}[v^{-1}]$;
     \item $([A]_{\Delta}, [A']_{\Delta'})_d=0$ for $(A,\Delta)\neq (A', \Delta)$.
  \end{enumerate}
\end{lem}
\begin{proof}
    Part~$(b)$ follows from the definition.
    For part $(a)$, we can realize $\sharp X^{f'}_{A^t, \Delta^t}$ as the coefficient of 
    $e_{D_{\co(A), \emptyset}}$ in the convolution of $e_{A^t, \Delta^t}\ast e_{A,\Delta}$.
    Thus we have $([A]_{\Delta}, [A']_{\Delta'})_d\in \mathbb{Z}[v,v^{-1}]$. 
    On the other hand, $X^{f'}_{A^t, \Delta^t}$ is an irreducible variety and we have seen that its dimension is $d_{A^t,\Delta^t}$. 
    Combining the Lang-Weil estimate, we have
    $$
    ([A]_{\Delta}, [A]_{\Delta})_d=q^{-d_{A^t,\Delta^t}}\sharp{X^{f'}_{A^t,\Delta^t}}\in 1+v^{-1}\mathbb{Z}[v^{-1}].
    $$
\end{proof}

For the basis 
$$\{A\}_{\Delta}=[A]_{\Delta}+\sum_{(A_1,\Delta_1)\prec (A,\Delta)} P_{(A, \Delta), (A',\Delta')}[A']_{\Delta'}.$$
We have the following proposition, which follows immediately from the orthogonality of the
standard basis established in Lemma~\ref{orth-f-stan}.

\begin{prop}\label{orth-f-c}
    For any $(A,\Delta), (A',\Delta')\in \Xi_{n|n,d}$, we have 
    $$
    (\{A\}_{\Delta}, \{A'\}_{\Delta'})_d\in \delta_{(A,\Delta), (A',\Delta')}+v^{-1}\mathbb{Z}[v^{-1}].
    $$
\end{prop}

For later use, we prove the following proposition for the inner product.
\begin{prop}\label{inner-transport}
    Let $(A,\Delta)\in \Xi_{n|n, d}$ and $h, h'\in \MS$. Then
    \begin{equation}\label{switch}
         ([A]_{\Delta}\ast h, h')_d=v^{d_{A,\Delta}-d_{A^t, \Delta^t}}(h, [A^t]_{\Delta^t}\ast h')_d.
    \end{equation}
\end{prop}
\begin{proof}
    By bilinearity, it is enough to take $h=e_{B, \Delta_B}$, $h'=e_{C,\Delta_C}$ where $(B,\Delta_B), (C,\Delta_C)\in \Xi_{n|n, d}$.
    Then both sides vanish unless $\co(A)=\ro(B)$, $\ro(A)=\ro(C)$ and $\co(B)=\co (C)$. 
    Assume these conditions from now on. 

    By definition
    \begin{align*}
       [A]_{\Delta}\ast e_{B,\Delta_B}&=v^{-d_{A,\Delta}}e_{A,\Delta}\ast e_{B, \Delta_B},\\   
       [A^t]_{\Delta^t}\ast e_{C,\Delta_C}&= v^{-d_{A^t,\Delta^t}}e_{A^t,\Delta^t}\ast e_{C, \Delta_C}.
    \end{align*}
    Then
    \begin{align*}
        &([A]_{\Delta}\ast e_{B,\Delta_B}, e_{C,\Delta_C})_d\\
        =&\sum_{f, f', \mu}v^{-d_{A,\Delta}}e_{A,\Delta}\ast e_{B, \Delta_B}(f, f',\mu)e_{C,\Delta_C}(f, f', \mu)\\
        =&\sum_{f, f', \mu}v^{-d_{A,\Delta}}(\sum_{f'', \mu'}e_{A,\Delta}(f, f'', \mu')e_{B,\Delta_B}(f'',f',\mu-\mu'))e_{C,\Delta_C}(f, f', \mu)\\
        =&v^{2d_{C,\Delta_C}-2d_{C^t, \Delta_C^t}-d_{A,\Delta}}\sharp X^{f'}_{C^t,\Delta_C^t}\\
        &\times \sharp
        \{(f'', \mu')\mid (f, f'', \mu')\in \mcal O_{A,\Delta}, (f'', f', \mu-\mu')\in \mcal{O}_{B, \Delta_B}\}\\
        =&v^{2d_{C,\Delta_c}-2d_{C^t, \Delta_c^t}-d_{A,\Delta}}
        \sharp  \mathfrak X^{(C,\Delta_c)}_{(A,\Delta), (B,\Delta_b)},
    \end{align*}
    where
    $$
    \mathfrak X^{(C,\Delta_C)}_{(A,\Delta), (B,\Delta_B)}=\{(f, f'', \mu, \mu')\ |\ (f, f'', \mu')\in \mcal O_{A,\Delta}, (f'', f', \mu-\mu')\in \mcal{O}_{B, \Delta_B}, (f, f',\mu)\in \mcal{O}_{C, \Delta_C}\}.
    $$

    On the other hand, 
    \begin{align*}
        &(e_{B,\Delta_B}, [A^t]_{\Delta^t}\ast e_{C,\Delta_C})_d\\
        =&\sum_{\tilde{f}, \tilde{f}', \nu}e_{B, \Delta_B}(\tilde{f}, \tilde{f}',\nu)v^{-d_{A^t,\Delta^t}}e_{A^t,\Delta^t}\ast e_{C,\Delta_C}(\tilde{f}, \tilde{f}', \nu)\\
        =&\sum_{\tilde{f}, \tilde{f}', \nu}e_{B,\Delta_B}(\tilde{f},\tilde{f}',\nu))v^{-d_{A^t,\Delta^t}}(\sum_{\tilde{f}'', \nu'}e_{A^t,\Delta^t}(\tilde{f}, \tilde{f}'', \nu')e_{C,\Delta_C}(\tilde{f}, \tilde{f}', \nu-\nu'))\\
        =&v^{2d_{B,\Delta_B}-2d_{B^t, \Delta_b^B}-d_{A,\Delta}}\sharp X^{\tilde{f}'}_{B^t,\Delta_b^t}\\
        &\times \sharp
        \{(\tilde{f}'', \nu')\ |\ (\tilde{f}, \tilde{f}'', \nu')\in \mcal O_{A^t,\Delta^t}, (\tilde{f}'', \tilde{f}', \nu-\nu')\in \mcal{O}_{C, \Delta_c}\}\\
        =&v^{2d_{B,\Delta_b}-2d_{B^t, \Delta_b^t}-d_{A^t,\Delta^t}}
        \sharp \mathfrak X^{(B,\Delta_b)}_{(A^t,\Delta^t), (C,\Delta_c)},
    \end{align*}
    where
    $$
   \mathfrak X^{(B,\Delta_B)}_{(A^t,\Delta^t), (C,\Delta_C)}=\{(\tilde{f}, \tilde{f}'', \nu, \nu')\ |\ (\tilde{f}, \tilde{f}'', \nu')\in \mcal O_{A^t,\Delta^t}, (\tilde{f}'', \tilde{f}', \nu-\nu')\in \mcal{O}_{C, \Delta_C}, (\tilde{f}, \tilde{f}',\nu)\in \mcal{O}_{B, \Delta_B}\}.
    $$

There is a bijection
   \begin{align*}
       \mathfrak X^{(C,\Delta_c)}_{(A,\Delta), (B,\Delta_b)}&\to  \mathfrak X^{(B,\Delta_B)}_{(A^t,\Delta^t), (C,\Delta_C)}\\
       (f,f'', \mu, \mu') &\mapsto (f'', f, \mu-\mu', -\mu' ).
   \end{align*}
    
    So identity \eqref{switch} holds if  
    $$
    2d_{C,\Delta_c}-2d_{C^t, \Delta_c^t}-d_{A,\Delta}=2d_{B,\Delta_b}-2d_{B^t, \Delta_b^t}+d_{A,\Delta}-2d_{A^t,\Delta^t},
    $$
    which follows by $\co(A)=\ro(B)$, $\ro(A)=\ro(C)$ and $\co(B)=\co(C)$.
\end{proof}

Recall that the algebra $\MS$ is generated by $E_i$, $F_i$, $H_a$ and $L_{D,a}$ for $i\in [1, n-1]$, $a\in [1, n]$ and $(D,\{a,a\})\in \Xi_{n|n,d}$.
The following corollary is an immediate consequence.
\begin{cor}
    Let $i\in [1, n-1]$ and $h, h'\in \MS$, we have
    \begin{enumerate}[$(a)$]
        \item $(E_ih, h')_d=(h, vK_iF_ih')_d$;
        \item $(F_ih, h')_d=(h, vK^{-1}_{i}E_ih')_d$;
        \item $(K_ih, h')_d=(h, K_ih')_d$;
        \item $(L_{D,a}h, h')_d=(h, L_{D,a}h')_d$,
    \end{enumerate}
    where $K_i=H_iH_{i+1}^{-1}$.
\end{cor}
\begin{proof}
    The identities $(a)-(c)$ can be reduced from the result in \cite{Mc12}. 

    For $(d)$, recall the definition of $L_{D,a}$, it is straightforward to check $(d)$ using Proposition \ref{inner-transport}.
\end{proof}

Recall the algebra homomorphism $\chi:S_{n, n}\to \mathcal{A}$ defined in \cite{FL21}, it maps $e_{A}$ to $\det(A)$ for each $A\in \Theta_n$.

Throughout the remainder of the paper, we distinguish basis elements
belonging to mirabolic $q$-Schur algebras of different $d$ by
adding the $d$ in the subscript. More precisely, for
$(A,\Delta)\in\Xi_{n|n,d}$, we denote the corresponding standard,
monomial, and canonical basis elements of $\mathcal{MS}_{n,d}$ by
\[
[A]_{\Delta,d},\quad
_{d}m_{A,\Delta},\quad
\{A\}_{\Delta,d},
\]
respectively. 
The notation without the subscript $d$ refers to the basis elements of $\mmuvgln$.


The transfer map $\phi^{\mathfrak{m}}_{d+n, d}:\mcal{MS}_{n,d+n}\to \MS$ is defined as the composition of the following homomorphisms.
$$
\mcal{MS}_{n,d+n}\xrightarrow{\rho} \mathcal{S}_{n,n}\otimes \mathcal{MS}_{n, d} \xrightarrow{\chi\otimes \mathrm{id}} \mcal A\otimes \mathcal{MS}_{n, d}.
$$
By definition, we have the following lemma which follows from \cite[Lemma 1.10]{L00} and \cite[Lemma 1.4.1]{FL21}.
\begin{lem}\label{t-map f gen}
For $1\leq i<n$, $1\leq h\leq n$ and $(D,\{(h, h)\})\in \Xi_{n|n,d+n}$, we have
    \begin{align*}
        \phi^{\mathfrak{m}}_{d+n, d}(E_i)&=E'_i;\\
        \phi^{\mathfrak{m}}_{d+n, d}(F_i)&=F'_i;\\
        \phi^{\mathfrak{m}}_{d+n, d}(K^{\pm 1}_i)&=K^{'\pm 1}_i;
    \end{align*}
\end{lem}
   
\begin{cor}\label{t-map f bas}
\begin{enumerate}[$(i)$]
    \item  For $(A,\emptyset)\in \Xi_{n|n, d+n}$ satisfying one of the following conditions:
   \begin{enumerate}[$(a)$]
       \item $A$ is a diagonal matrix;
       \item $A-E_{i, i+1}$ is a diagonal matrix for some $i\in [1, n-1]$;
       \item $A-E_{i+1, i}$ is a diagonal matrix for some $i\in [1, n-1]$.
   \end{enumerate}
   We have
   $$
   \phi^\m_{d+n, d}([A]_{\emptyset, d+n})=
   \begin{cases}
       [A-I]_{\Delta, d} & \text{ if } (A-I,\emptyset)\in \Xi_{n|n, d};\\
       0 & \text{ otherwise }.
   \end{cases}
   $$
    \item For $1\leq h\leq n$ and $(D,\{(h,h)\}\in \Xi_{n|n,d+n}$, where $D$ is a diagonal matrix, we have
     $$
   \phi^\m_{d+n, d}(L_{D,h})=
   \begin{cases}
        v^{h}L_{D-I, h} & \text{ if } (D-I,\{(h, h)\})\in \Xi_{n|n, d};\\
        0 & \text{ if } (D-I,\{(h, h)\})\notin \Xi_{n|n, d}.
    \end{cases}
   $$
\end{enumerate}
\end{cor}
\begin{proof}
    Part $(i)$ follows Lemma \ref{t-map f gen}.
    
    For part $(ii)$, suppose that $(D,\{h, h\})\in \Xi_{n|n, d+n}$ and $d_h>1$ otherwise both sides are zero.
    We have $\rho(L_{D, h})=1\otimes v^hL_{D-I,  h)}$. 
    Then we apply $\chi\otimes \mathrm{id}$, we have $\phi^\m_{d+n, d}(L_{D,h})=v^{h} L_{D-I, h}$.
\end{proof}

\subsection{Modified Mirabolic Quantum $\slln$ and its Canonical Basis}

Recall that in \cite{R18} and \cite{GR25}, 
the algebra $\muvsln$ is defined as the subalgebra of $\muvgln$ generated by $E_i$, $F_i$, $L$ and $K_{i}^{\pm 1}=(H_iH_{i+1}^{-1})^{\pm 1}$ for $1\leq i\leq n-1$ with the following relations.
 \begin{align*}
  K_i^{\pm1}K_i^{\mp1}=1, & \quad K_iK_j=K_jK_i; \\
 K_iE_jK_i^{-1} &= v^{2\delta_{ij}-\delta_{i,j+1}-\delta_{i+1, j}}E_j;\\
K_iF_jK_i^{-1} &= v^{-2\delta_{ij}+\delta_{i,j+1}+\delta_{i+1, j}}F_j;\\
 E_iF_j-F_jE_i & =  \delta_{ij}\frac{K_i-K_i^{-1}}{v-v^{-1}};\\
 E_iE_j=E_jE_i, & \quad F_iF_j=F_jF_i \quad\text{ if }|i-j| > 1 ;\\
 (v+v^{-1})E_iE_jE_i & =  E_i^2E_j + E_jE_i^2 \quad\text{ if }|i-j| = 1;\\
 (v+v^{-1})F_iF_jF_i & =  F_i^2F_j + F_jF_i^2 \quad\text{ if }|i-j| = 1;\\
 L^2=L,&\quad K_iL=L K_i;\\
 LE_i=E_iL,&\quad LF_i=F_iL \quad\text{ if }i\geq 2;\\
L E_1L=L E_1,&\quad L F_1L=F_1L;\\
 (v+v^{-1})E_1LE_1 &= v^{-1}E_1^2L + vL E_1^2;\\
(v+v^{-1})F_1L F_1 &= v^{-1}LF_1^2 + vF_1^2L.
\end{align*}
Consider the equivalence relation on $\mbb Z^n$ defined by 
$(\lambda_1,\cdots, \lambda_n)\thicksim (\mu_1,\cdots, \mu_n)$ iff 
$$
(\lambda_1-\mu_1,\cdots, \lambda_n-\mu_n)=k(1, \cdots, 1).
$$
for some integer $k$. The set of equivalence classes is denoted by $\mbb X=\mbb Z^n/\thicksim$ and 
the equivalence class of $\lambda\in \mbb Z^n$ is denoted by $\bar{\lambda}$.

For any $\bar{\lambda}, \bar{\mu}\in \mbb X$, we define
$$
_{\bar{\lambda}}\muvsln_{\bar{\mu}}=\muvsln/(\sum_{i=1}^{n-1}(K_i-v^{\lambda_i-\lambda_{i+1}})\muvsln+\sum_{i=1}^{n-1}\muvsln(K_i-v^{\mu_i-\mu_{i+1}})),
$$
and the modified mirabolic quantum $\slln$ as 
$$
\mmuvsln=\bigoplus_{\bar{\lambda},\bar{\mu}\in \mbb X}{}_{\bar{\lambda}}\muvsln_{\bar{\mu}}.
$$
There is a projection $\pi^{\m}_{\bar{\lambda}, \bar{\mu}}:\muvsln\to {}_{\bar{\lambda}}\muvsln_{\bar{\mu}}$ and the multiplication is given by:
For any $\lam, \lam', \muu, \muu'\in \mbb X$ and $t\in _{\bar{\lambda}}\muvsln_{\bar{\mu}}, s\in _{\bar{\lambda'}}\muvsln_{\bar{\mu}'}$, 
$$
\pi^{\m}_{\lam, \muu}(t)\pi^{\m}_{\lam',\muu'}(s)=
\begin{cases}
\pi^{\m}_{\lam, \muu'}(ts) & \text{ if } \muu=\lam',\\
0 &\text{ otherwise}.
\end{cases}
$$ 

Now we define $1_{\bar{\lambda}}=\pi^{\m}_{\bar{\lambda},\bar{\lambda}}(1)$. 
Then for any $\bar{\lambda}\in \mbb X$, the element $1_{\bar{\lambda}}$ is idempotent. 
We have 
$$
{}_{\bar{\lambda}}\muvsln_{\bar{\mu}}=1_{\bar{\lambda}}\mmuvsln1_{\bar{\mu}}.
$$
So we have 
$$
\mmuvsln=\bigoplus_{\bar{\lambda},\bar{\mu}\in \mbb X} 1_{\bar{\lambda}}\mmuvsln1_{\bar{\mu}}.
$$

Recall that for any $\lambda\in \mbb Z^n$, $|\lambda|=\sum_{i}\lambda_i$. For any $\lam\in \mbb X$, 
we define $|\lam|=|\lambda|\mod n$, which is independent of the choice of representative $\lambda$.

For $\bar d\in\mathbb Z/n\mathbb Z$, let
\[
\mmuvsln\langle\bar d\rangle
=
\bigoplus_{\substack{
\bar\lambda,\bar\mu\in\mathbb X\\ |\bar{\lambda}|=|\bar{\mu}|= \bar{d}}}
1_{\bar\lambda}\mmuvsln1_{\bar\mu}.
\]
Then
$$
\mmuvsln=\bigoplus_{\bar{d}\in \mathbb Z/n\mbb Z} \mmuvsln\langle \bar{d}\rangle.
$$

By definition, the algebra $\mmuvsln$ is a $\muvsln$-bimodule, therefore, the elements $E_i1_{\lam}$, $ F_i1_{\lam}$ and $L1_{\lam}$ are meaningful for any $\lam\in \mbb X$. For any $d\in \mbb Z$, there exists an algebra isomorphism
$$
\varrho_{d}: \mmuvgln\langle d\rangle \to \mmuvsln\langle\bar{d}\rangle,
$$
by sending $1_{\lambda}\mapsto 1_{\lam}$ $E_i1_{\lambda}\mapsto E_i1_{\lam}$, $F_i1_{\lambda}\mapsto F_i1_{\lam}$ and $L1_{\lambda}\mapsto L1_{\lam}$ for all $1\leq i\leq n-1$ and $\lambda\in \mbb Z^n$ with $|\lambda|=d$.

With these isomorphisms, we have the $\mbb Q(v)$-linear map 
$$
\varrho: \mmuvgln \to \mmuvsln.
$$

We define the surjective algebra homomorphism $\phi^{\mathfrak{m}}_d:\mmuvsln\to \MS$ as the composition of
$$
\mmuvsln\xrightarrow{\pi^{\m}_{\bar{d}}}\mmuvsln\langle\bar{d}\rangle\xrightarrow{\cong}\mmuvgln\langle d\rangle\xrightarrow{\Phi_d} \MS.
$$
where $\pi^{\m}_{\bar{d}}$ is the projection from $\mmuvsln$ to $\mmuvsln\langle\bar{d}\rangle$.

For $p\in \mbb Z$, we define an algebra isomorphism $\xi_{p}^{\m}:\mmuvgln\to \mmuvgln$ by 
sending 
\begin{align*}
    E_i1_{\lambda}&\mapsto  E_i1_{\lambda+p(1,\cdots, 1)},\\
    F_i1_{\lambda}&\mapsto  F_i1_{\lambda+p(1,\cdots, 1)},\\
    1_{\lambda}&\mapsto  1_{\lambda+p(1,\cdots, 1)},\\
    L1_{\lambda}&\mapsto  L1_{\lambda+p(1,\cdots, 1)},
\end{align*}
for each $\lambda\in \mbb Z^n$.
Then by definition, we have $\xi^{\m}_p(L_{D, h})=v^{-ph}L_{D+pI, h}$ for each $(D, \{(h, h)\})\in \widetilde{\Xi}$ where $D$ is a diagonal matrix.

 By definition, we have 
 \begin{align}\label{shift-rhod}
         \varrho_{d+np}\circ \xi^{\m}_{p}=\varrho_{d}.
    \end{align}
    and 
    \begin{align}\label{shift-rho}
         \varrho\circ \xi^{\m}_p=\varrho \text{ for any } p\in \mbb Z.
    \end{align}

We first prove that 
\begin{lem}\label{theta-equal}
    Let $(A,\Delta)\in \widetilde{\Xi}$ with $|A|=d_0$ and set $d=d_0+np$, then for large enough $p$ we have 
    \begin{align*}
        {}_{d}m_{_pA,\Delta}=\{_pA\}_{\Delta,d}+\sum_{\substack{(A',\Delta')\prec (A,\Delta)\\ \vartheta(\Delta')=\vartheta(\Delta)}}\omega_{(_pA,\Delta), ({}_pA',\Delta')}\{_pA'\}_{\Delta',d}.
    \end{align*}
\end{lem}
\begin{proof}
    All varieties considered below are defined over \(\mathbb F_q\), and we implicitly
    pass to \(\overline{\mbb F}_q\) when discussing their geometric properties, orbit closures,
    dimensions, and intersection cohomology. We omit the Tate twist in the intersection cohomology.

    For convenience, we set $\vartheta(\Delta)=h$. By definition, there exist $({}_{p}B_r, \emptyset)\in \Xi_{n|n,d}$ for $r=1,\ldots,N$ 
    and $L_{D, h}$ such that
    $$
    {}_{d}m_{_pA,\Delta}=[{}_{p}B_1]_{\emptyset}\ast\cdots L_{D,\vartheta(\Delta)}\ast\cdots \ast [_pB_{N}]_{\emptyset},
    $$
    where $_pB_r$ has the form ${}_{p}D_r+a_{ij}E_{r, r+1}$ or ${}_{p}D_r+a_{ij}E_{r+1, r}$ for some proper diagonal matrix $D_r$ and elements $a_{ij}$.
    
    Fix this expression, 
    we define a variety of flags
    $$
     Z_{{}_{p}A}=\{(f^{(1)},\cdots , f^{(N+1)})\mid (f^{(i)},f^{(i+1)})\in \mcal O_{{}_{p}B_i, \emptyset}\}
    $$
    and its closed subvariety
    $$
    Y_{{}_{p}A,\Delta}=\{(z,w)\in Z_{{}_{p}A}\times V\mid w\in V^{(l)}_{h}(z)\},
    $$
    $V^{(l)}_{h}(z)$ is the $h$-component of the flag $f^{(l)}$ in $z$.
    Following the definition, we can see that $ Z_{{}_{p}A}$ is an irreducible smooth variety and 
    $Y_{{}_{p}A,\Delta}$ is a vector bundle with base space $Z_{{}_{p}A}$ and 
    its rank is $s_p=\dim V^{(l)}_{h}(z)=\sum_{i=1}^h\ro ({}_{p}B_{l})_i=\sum_{i=1}^h\ro (B_{l})_i+hp$. 
    Therefore, $Y_{{}_{p}A,\Delta}$ is also an irreducible smooth variety. 

    We have the proper morphism
    $$
    \pi_p:Y_{{}_{p}A,\Delta}\to \X\times \X\times V,\quad (z, w)\mapsto (f, g, w)
    $$
    and the image is $\overline{\mcal O_{{}_{p}A,\Delta}}$.
    Then Grothendieck's sheaf-function correspondence shows that $_{d}m_{{}_{p}A,\Delta}$ corresponds to 
    $\mcal F_{_{}pA,\Delta}=R\pi_{p, \ast}(\underline{\bar{\mbb Q}_{l}}_{Y_{{}_{p}A,\Delta}})[\dim Y_{{}_{p}A,\Delta}][-r(_{p}A,\Delta)]$.

    Since $\pi_p$ is proper, we apply the Decomposition theorem and obtain
    $$
    \mcal F_{_{}pA,\Delta}\cong\bigoplus_{\substack{({}_{p}A',\Delta')\preceq (A,\Delta)\\j}}IC_{\mcal O_{{}_{p}A',\Delta'}}[j-r(_{p}A,\Delta)]^{\oplus m_{{}_{p}A',\Delta'}}.
    $$
    Applying the cohomology functor, we obtain
    $$
    \mcal H^k(\mcal F_{_{}pA,\Delta})=\bigoplus_{\substack{({}_{p}A',\Delta')\preceq (_pA,\Delta)\\j}}\mcal H^k( IC_{\mcal O_{{}_{p}A',\Delta'}}[j-r(_{p}A,\Delta)]^{\oplus m_{{}_{p}A',\Delta'}}).
    $$
    
    For any fixed summand $IC_{\mcal O_{{}_{p}A',\Delta'}}[j-r(_{p}A,\Delta)]$ and $(f', g',w')\in \mcal O_{_{p}A',\Delta'}$, 
    the proper base change shows
    $$
     R\pi_{p, \ast}(\underline{\bar{\mbb Q}_{\ell}}_{Y_{{}_{p}A,\Delta}})_{(f', g', w')}\cong R\Gamma ((Y_{{}_{p}A,\Delta})_{(f',g',w')}, \bar{\mbb Q}_\ell),
    $$
    $(Y_{{}_{p}A,\Delta})_{(f',g',w')}$ is the fiber of $(f', g', w')$ under $\pi_p$. Thus we have
    $$
    \mcal H^k(\mcal F_{_{}pA,\Delta})_{(f', g', w')}\cong H^{k+\dim Y_{{}_{p}A,\Delta}-r(_{p}A,\Delta)}((Y_{{}_{p}A,\Delta})_{(f',g',w')}, \bar{\mbb Q}_\ell).
    $$
    Thus $\mcal H^k(\mcal F_{_{}pA,\Delta})_{(f', g', w')}=0$ if $k> 2\dim (Y_{{}_{p}A,\Delta})_{(f',g',w')}-\dim Y_{{}_{p}A,\Delta}+r(_{p}A,\Delta)$. 
    Moreover, the definition shows $\mcal H^k( IC_{\mcal O_{{}_{p}A',\Delta'}}[j-r(_{p}A,\Delta)])_{(f', g',w')}\neq 0$ 
    if and only if $k=-\dim \mcal O_{_{p}A', \Delta'}-j+r(_{p}A,\Delta)$.
    Combining these conditions and Verdier duality, we have
    $$
    |j|\leq 2\dim (Y_{{}_{p}A,\Delta})_{(f',g',w')}-\dim Y_{{}_{p}A,\Delta}+\dim \mcal O_{_pA', \Delta'},
    $$
    Since $r(_{p}A,\Delta)=r(_{p}A)=r(_{p}A',\Delta')$ for any $(_{p}A',\Delta')\preceq (_{p}A,\Delta)$, 
    $$
    |j|\leq 2\dim (Y_{{}_{p}A,\Delta})_{(f',g',w')}-\dim Y_{{}_{p}A,\Delta}+d_{_pA', \Delta'}.
    $$
    
    The next step is to estimate the dimension of the above spaces.
    As a vector bundle, the dimension of $ Y_{{}_{p}A,\Delta}$ equals to $\dim Z_{{}_{p}A}+s_p$.
    By \cite{FZM25}[Corollary 3.2.3], 
    $$\dim \mcal O_{_pA', \Delta'}=\dim \mcal O_{_pA'}+d(\Delta')+p\vartheta(\Delta'),$$
    where $d(\Delta')=\sum_{\{(i,j)\}\leq \Delta'}a'_{ij}$. Then 
    $$
    d_{_{p}A',\Delta'}=d_{_{p}A'}+d(\Delta')+p\vartheta(\Delta').
    $$
    For $\dim (Y_{{}_{p}A,\Delta})_{(f',g',w')}$, we consider the subvariety $Z_{{}_pA, (f',g')}$ of $Z_{{}_pA}$ defined as
    $$
    \{ (f^{(1)},\cdots, f^{(h)}\cdots , f^{(N+1)})\in Z_{{}_pA}\mid f^{(1)}=f', f^{(N+1)}=g'\}.
    $$
    and a vector bundle $(Y_{{}_pA,\Delta})_{(f', g')}=\{(z, w)\in Z_{{}_pA}\times V\mid w\in V^{(h)}_h(z)\}$. We have
    $$\dim (Y_{{}_pA,\Delta})_{(f', g')}=\dim Z_{{}_pA, (f',g')}+s_p.$$ 
    Moreover, $(Y_{{}_pA,\Delta})_{(f',g',w')}$ can be regarded as the fiber of $w'$ under the second projection $pr_2:(Y_{{}_pA,\Delta})_{(f', g')}\to V$.
    However, the fibers for different vectors may not have the same dimension. So we restrict the projection to $\{(z,w) \mid w\in G_{(f',g')}\cdot w'\}$, where
    $G_{(f',g')}$ is the stabilizer of $(f', g')$. The dimension of $G_{(f',g')}\cdot w'$ is $\sum_{\{(i, j)\}\leq \Delta'}a'_{ij}+p\vartheta(\Delta')$
    Then we have
    $$
    \dim (Y_{{}_pA,\Delta})_{(f',g',w')}+d(\Delta')+p\vartheta(\Delta')\leq \dim Z_{{}_pA, (f', g')}+s_p.
    $$
    Since $\dim (Y_{{}_pA,\Delta})_{(f',g',w')}\leq \dim Z_{{}_pA, (f',g')}$, we have
    $$
    \dim (Y_{{}_pA,\Delta})_{(f',g',w')}\leq \dim Z_{{}_pA, (f', g')}+\min \{0, s_p-d(\Delta')-p\vartheta(\Delta')\}.
    $$

    Combining above results, we have 
    $$
    |j|\leq 2\dim Z_{{}_pA, (f', g')}-\dim Z_{{}_{p}A}+d_{_pA'}+2\min \{0, s_p-d(\Delta')-p\vartheta(\Delta')\}-(s_p-d(\Delta')-p\vartheta(\Delta'))
    $$
    which shows
    $$
    |j|\leq 2\dim Z_{{}_pA, (f', g')}-\dim Z_{{}_{p}A}+d_{_pA'}+|d(\Delta')-\sum_{i=1}^h\ro (B_{l})_i|-|p(\vartheta(\Delta')-h)|.
    $$
    Then the result holds if there exists a $K\in \mbb N$ such that
    $$
    2\dim Z_{{}_pA, (f', g')}-\dim Z_{{}_{p}A}+d_{_pA'}\leq K
    $$
    for each $p$. 
    If such $K$ exists, we have 
    $$
    |j|\leq K+|d(\Delta')-\sum_{i=1}^h\ro (B_{l})_i|-|p(\vartheta(\Delta')-h)|,
    $$
    which means if there is $IC_{\mcal O_{(_pA',\Delta')}}$ in the direct sum and $\vartheta(\Delta')\neq h$, 
    then for large enough $p$, we have $K+|d(\Delta')-\sum_{i=1}^h\ro (B_{l})_i|-|p(\vartheta(\Delta')-h)|<0$ which is absurd, and the lemma holds.

    To find such $K$, we consider the elements
    $$
    {}_{d}m_{_pA}=[{}_{p}B_1]\ast\cdots \ast[_pB_{N}]_{\emptyset}.
    $$
    Let $d_{A}=d(A,\emptyset)-r(A,\emptyset)=\sum_{i\leq k, j>l}a_{ij}a_{kl}$ for each $A\in \Theta_d$, 
    and $N_{_{p}A'}(q)=\sharp Z_{_{p}A, (f', g')}(\mbb F_q)$ for $_{p}A'\in \Theta_d$ and $(f', g')\in \mcal O_{_{p}A'}$.
    On the one hand, we have 
    \begin{align*}
        {}_{d}m_{_pA}(f', g')&=v^{-\sum_{k=1}^{N}d_{{}_pB_k}}\sharp Z_{{}_{p}A, (f', g')}(\mbb F_q)\\
        &=v^{-\dim Z_{{}_{p}A}+r(_{p}A)}N_{_{p}A'}(q).
    \end{align*}
    On the other hand, we can write ${}_{d}m_{_pA}=\sum_{_{p}M\in \Theta_d}c_{_{p}M}(v)[_{p}M]$,
    then ${}_{d}m_{_pA}(f',g')=c_{_{p}A'}(v)v^{-d_{A'}}$. Combine these two identities, we have
    $$
    c_{_{p}A'}(v)=v^{d_{A'}-\dim Z_{{}_{p}A}+r(_{p}A,\Delta)}N_{_{p}A'}(v^2).
    $$
    By construction, the degree of $N_{_{p}A'}(q)$ is $\dim Z_{_{p}A, (f', g')}$.
    Thus the degree of $ c_{_{p}A'}(v)$ is
    $$
    \deg c_{_{p}A'}(v)=d_{A'}-\dim Z_{{}_{p}A}+2\dim Z_{_{p}A, (f', g')}+r(_{p}A,\Delta).
    $$
    Recall the stabilization procedure in \cite[Proposition 4.2]{BLM90}, 
    we have $G_i(v, v')\in \mbb Q(v)[v']$ such that
    $$
     {}_{d}m_{_pA}=\sum_{i=1}^{r}G_{i}(v,v^{-p})[_{p}M_i].
    $$

    By definition 
    $$
    G_i(v, v')=\sum_{k=0}^{r_i}g_{i, k}(v)(v')^{k},
    $$
    where $g_{i, k}(v)$ are fixed elements in $\mbb Q(v)$ and independent of $p$. Then
    \begin{align*}
        \deg G_{i}(v,v^{-p})\leq \max_{0\leq k\leq r_i}\{\deg g_{i, k}(v)-kp\}\leq \max_{0\leq k\leq r_i}\{\deg g_{i, k}(v)\}.
    \end{align*}
    
    Take 
    $$
    K=\max \left\{0, \max_{\substack{1\leq i\leq r\\0\leq k\leq r_i}} \deg g_{i, k}\right \}
    $$
    which is independent of $p$ and we have
    $$
    \deg c_{_{p}A'}(v)\leq K
    $$
    as desired.
\end{proof}
\begin{lem}\label{conv mono}
      Let $(A,\Delta), (A',\Delta')\in \widetilde{\Xi}$ with $|A|=|A'|=d_0$, and set $d=d_0+np$, then 
        $$
        (_{d}m_{_pA,\Delta}, {}_{d}m_{_pA',\Delta'})_d
        $$
        converges in $\mbb Q((v^{-1}))$ to an element in $\mbb Q(v)$ as $p\to \infty$
\end{lem}
\begin{proof}
    By Proposition \ref{stab}, there exist $ G_{\Delta_i}\in \mbb Q(v)[v',v'^{-1}]$ and $(_{p}Z_i,\Delta_i)\in \Xi_{n|n,d}$ such that
    $$
    _{d}m_{_{p}A,\Delta}=\sum_{i=1}^{m} G_{\Delta_i}(v, v^{-p})[_{p}Z_{i}]_{\Delta_i}.
    $$
    Similarly, we have
     $$
    _{d}m_{_{p}A',\Delta'}=\sum_{i=1}^{m'} G'_{\Delta'_i}(v, v^{-p})[_{p}Z'_{i}]_{\Delta'_i}.
    $$
    Then the proof of Lemma \ref{theta-equal} shows that
    for large enough $p$, there is a $K\in \mbb N$ which is independent of $p$, such that
    $\deg G_{\Delta_i}(v, v^{-p})\leq K$ and  $\deg G'_{\Delta'_i}(v, v^{-p})\leq K$ for all $i$.
    Therefore, we have $G_{\Delta_i}(v, v'), G'_{\Delta'_i}(v, v')\in \mbb Q(v)[v']$. 
    Moreover, $G_{\Delta_i}(v, v')$ and $G_{\Delta'_i}(v, v')$ can be written as
    \begin{align*}
        G_{\Delta_i}(v, v')&=a_{i,0}(v)+\sum_{t=1}^ra_{i, t}(v)(v')^t, \\
        G'_{\Delta'_i}(v, v')&=a'_{i, 0}(v)+\sum_{t=1}^{r'}a'_{i, t}(v)(v')^t,
    \end{align*}
    where $a_{i, 0}(v)$ and $a'_{i,0}(v)$ are independent of $p$.

    Then we have
    \begin{align*}
      (_{d}m_{_pA,\Delta}, {}_{d}m_{_pA',\Delta'})_d=\sum_{i, j} G_{\Delta_i}(v, v^{-p})G'_{\Delta'_j}(v, v^{-p})([_{p}Z_{i}]_{\Delta_i}, [_{p}Z'_{j}]_{\Delta'_j})_d.
    \end{align*}
    Next, we have
    \begin{align*}
        ([_{p}Z_{i}]_{\Delta_i}, [_{p}Z'_{j}]_{\Delta'_j})_d=v^{d_{_{p}Z_i,\Delta_i}-d_{{}_pZ_i^{t}, \Delta_i^{t}}}([{}_pD]_{\emptyset}, [{}_pZ_i^t]_{\Delta_i^t}\ast [{}_pZ'_j]_{\Delta'_j})_d.
    \end{align*}
    where $D$ is a diagonal matrix with $\ro({}_pD)=\co({}_pA)$. 
    
    By definition, $d_{_{p}Z_i,\Delta_i}-d_{{}_pZ_i^{t}, \Delta_i^{t}}$ is independent of $p$. Then apply \ref{stab}, 
    we have $([{}_pD]_{\emptyset}, [{}_pZ_i^t]_{\Delta_i^t}\ast [{}_pZ'_j]_{\Delta'_j})_d=H_{\Delta, \Delta'}(v, v^{-p})$ for some fixed $H_{\Delta, \Delta'}(v, v')\in \mathbb{Q}(v)[v', v'^{-1} ]$. 
    
    On the other hand, Lemma \ref{orth-f-stan} shows that $v^{d_{_{p}Z_i,\Delta_i}-d_{{}_pZ_i^{t}, \Delta_i^{t}}}H_{\Delta, \Delta'}(v, v^{-p})\in 1+v^{-1}\mbb Z[v^{-1}]$. 
    Then $v^{d_{_{p}Z_i,\Delta_i}-d_{{}_pZ_i^{t}, \Delta_i^{t}}}H_{\Delta, \Delta'}(v, v')\in \mathbb{Q}(v)[v']$ and when $p\to \infty$, we have  $v^{d_{_{p}Z_i,\Delta_i}-d_{{}_pZ_i^{t}, \Delta_i^{t}}}H_{\Delta, \Delta'}(v, v^{-p})$ converges to an element in $\mbb Q(v)$.
    
    Therefore, we have $(_{d}m_{_pA,\Delta}, {}_{d}m_{_pA',\Delta'})_d$ converges to an element in $\mbb Q(v)$.
\end{proof}

\begin{prop}
    Let $(A,\Delta)\in \widetilde{\Xi}$ with $|A|=d_0$, $\vartheta(\Delta)=h$ and set $d=d_0+np$, then for all large enough $p$ we have 
    $$
    \phi^{\m}_{d, d-n}(\{_{p}A\}_{\Delta, d})=v^{h}\{_{p-1}A\}_{\Delta, d-n}.
    $$
\end{prop}
\begin{proof}
    The proof is an adaptation of the arguments in \cite{Mc12,LW18} to the mirabolic setting. We include the
    details for completeness.
    
    Consider the monomial basis $\{_{d}m_{A,\Delta} \ |\ (A,\Delta)\in \Xi_{n|n,d} \}$ of $\MS$. 
    By Corollary\ref{t-map f bas}, we have
    $$
    \phi_{d, d-n}^{\m}({}_{d}m_{A,\Delta})=v^{\vartheta(\Delta)}{}_{d-n}m_{A-I,\Delta}
    $$
    where $_{d-n}m_{A-I,\Delta}=0$ if $(A-I, \Delta)\notin \Xi_{n|n, d-n}$.
    Therefore, we only need to prove that: 
    Let $(A,\Delta)\in \widetilde{\Xi}$ for all large enough integers $p$, we have
    \begin{align}
        \{_pA\}_{\Delta,d}={}_dm_{{}_pA,\Delta}+\sum_{\substack{(A',\Delta')\prec (A,\Delta)\\ \vartheta(\Delta')=\vartheta(\Delta)}}c_{(A',\Delta'), (A,\Delta)}{}_dm_{{}_pA',\Delta'},\label{cb-mb}
    \end{align}
    where $c_{(A',\Delta'), (A,\Delta)}$ is independent of $p$.

    By Lemma \ref{theta-equal}, we have
    \begin{align*}
        {}_{d}m_{_pA,\Delta}=\{_pA\}_{\Delta,d}+\sum_{\substack{(A',\Delta')\prec(A,\Delta)\\ \vartheta(\Delta')=\vartheta(\Delta)}}\omega_{(_pA,\Delta), ({}_pA',\Delta')}\{_pA'\}_{\Delta',d}.
    \end{align*}
    
    We shall prove \eqref{cb-mb} by mathematical induction on the partial order $\leq$.
    When $(A,\Delta)$ is minimal, we have $ {}_{d}m_{_pA,\Delta}=\{_pA\}_{\Delta,d}$ for all $p$ as desired.
    Now suppose \eqref{cb-mb} holds for all $(B,\Delta_B)\in \widetilde{\Xi}$ such that $(B,\Delta_B)\prec (A,\Delta)$.

    Let 
    $$
    I_d=\left\{(B,\Delta_B)\in \widetilde{\Xi}\ |\ \substack{(_pB,\Delta_B)\in \Xi_{n|n, d}, \\\co(B)=\co(A), \ro(B)=\ro(A),\\ (B,\Delta_B)\preceq (_{p}A,\Delta)\\ \vartheta(\Delta_B)=\vartheta(\Delta)=h}\right\}.
    $$
    Then for all large enough integer $p$, we have 
    $$
    I_d=\{(B,\Delta_B)\in \widetilde{\Xi}\ |\ \co(B)=\co(A), \ro(B)=\ro(A), (B,\Delta_B)\preceq (_{p}A,\Delta), \vartheta(\Delta_B)=h\}.
    $$
    This fact can be obtained by a similar argument in the proof of \cite[Lemma 5.1]{LW18} and $I_d$ is a finite set that is independent of $p\gg 0$.

    For any element $u\in \mcal A$, $\deg (u)$ denotes its degree. Let $x$ be an $\mcal A$-linear combination of $\{\{_pB\}_{\Delta_B,d}\ |\ (B,\Delta_B)\in I_d\}$,
    we set 
    $$
    \mrm n(x)=\max \{\deg ( x, \{_pB\}_{\Delta_B,d})_d\ |\ \{_pB\}_{\Delta_B,d}\in I_d, (B, \Delta_B)\neq (A,\Delta)\}
    $$
    and $\mrm n_p=\mrm n(_{d}m_{_pA,\Delta})$.
    
    If $\mrm n_p<0$ we set $b_{_pA,\Delta}= {}_{d}m_{_pA,\Delta}$. 
    Now suppose that $\mrm n_p\geq 0$, we define
    $$
     J_d=\{(B,\Delta_B)\in I_d\setminus{\{(A,\Delta)\}}\ |\ \deg ({}_{d}m_{{}_pA,\Delta}, \{_pB\}_{\Delta_B,d})_d=\mrm n_p\}.
    $$
    For each $(B,\Delta_B)\in I_d$, 
    $$
     ( _{d}m_{{}_pA,\Delta}, \{_pB\}_{\Delta_B,d})_d=\sum_{i\leq n_p}c_{B,\Delta_B,p, i}v^{i}\in \mathcal{A},
    $$
    where $c_{B,\Delta_B,p, i}\in \mathbb Z$ and $c_{B,\Delta_B,p, \mrm n_p}\begin{cases}
        \neq 0 & \text{ if } (B,\Delta_B)\in J_d,\\
        =0 & \text{ otherwise }.
    \end{cases}$

    Therefore, we can define a new element $_{d}m'_{_pA,\Delta}$ in $\MS$ which is bar-invariant as
    $$
    _{d}m'_{_pA,\Delta}=
    \begin{cases}
      _{d}m_{_pA,\Delta}-\sum_{(B,\Delta_B)\in J_d} c_{B,\Delta_B,p, \mrm n_p}(v^{n_p}+v^{-n_p})\{_pB\}_{\Delta_B,d} & \text{ if } \mrm n_p>0,\\
      _{d}m_{_pA,\Delta}-\sum_{(B,\Delta_B)\in J_d} c_{B,\Delta_B,p,\mrm n_p}\{_pB\}_{\Delta_B,d} & \text{ if } \mrm n_p=0.
    \end{cases}
    $$
    Now we investigate the relationship between $\mrm n(_{d}m'_{_pA,\Delta})$ and $\mrm n_p$. For $\mrm n_p>0$
    \begin{align*}
        (& _{d}m'_{_pA,\Delta}, \{_pB\}_{\Delta_B, d})_d\\=& ({}_{d}m_{_pA,\Delta}, \{_pB\}_{\Delta_B, d})_d
        - (\sum_{(B',\Delta_{B'})\in J_d}c_{B',\Delta_{B'},p, \mrm n_p}(v^{\mrm n_p}+v^{-\mrm n_p})\{_pB'\}_{\Delta_{B'},d}, \{_pB\}_{\Delta_{B}, d})_d\\
        \equiv& \sum_{i\leq \mrm n_p-1}c_{B,\Delta_B,p, i}v^{i}-\sum_{\substack{(B',\Delta_{B'})\in J_d\\
        (B',\Delta_{B'})\neq(B,\Delta_B)}}
        c_{(B',\Delta_{B'}),p,n_p}v^{n_p}
        ( \bigl\{{}_pB'\bigr\}_{\Delta_{B'},d}, \bigl\{{}_pB\bigr\}_{\Delta_B,d})_d
        \mod{v^{-1}\mathbb Z[v^{-1}]}.
    \end{align*}
    The calculation for $\mrm n_p=0$ is similar. Then using Proposition \ref{orth-f-c}, we have $\mrm n(_{d}m'_{_pA,\Delta})<n_p$.

    Repeating the above procedure for $_{d}m'_{_pA,\Delta}$, we will obtain a bar-invariant element $_{d}m''_{_pA,\Delta}$ with $\mrm n(_{d}m''_{_pA,\Delta})<\mrm n(_{d}m'_{_pA,\Delta})$. 
    We can repeat this process again and again to produce new bar-invariant elements. 
    Since $\mrm n_p\geq 0$, after finitely many steps we can obtain a bar-invariant element $b_{_pA, \Delta}$ that $n(b_{_pA,\Delta})<0$.
    
    By the construction, we have
    $$
    b_{_pA,\Delta}=\{_pA\}_{\Delta,d}+\sum_{(_{p}A,\Delta)\neq (B,\Delta_B)\in I_d}f_{B,\Delta_B}\{_pB\}_{\Delta_B,d}
    $$
    where $f_{B, \Delta_B}\in \mcal A$ and $\overline{f_{B, \Delta_B}}=f_{B, \Delta_B}$. However, if there is a $f_{B, \Delta_B}\neq 0$, then $\mrm n(b_{_pA,\Delta})\geq 0$ which contradicts our construction. Therefore, we have $ b_{_pA,\Delta}=\{_pA\}_{\Delta,d}$.

    In the process, the construction only depends on the first $\mrm n_p$ coefficients of $ ( _{d}m_{_pA,\Delta}, \{_pB\}_{\Delta_B,d})_d$ 
    and those of $( \{_pB'\}_{\Delta_{B'},d}, \{_pB\}_{\Delta_B,d})_d$ for $(B,\Delta_B)\in I_d$ and $(B',\Delta_{B'})\in J_d$. 
    We have seen that $I_d$ is independent of $p\gg 0$.
    So we need to prove that $\mrm n_p\geq 0$ and for $0\leq i\leq \mrm n_p$, $c_{B,\Delta_B, p, i}$ are independent of $p\gg 0$.

    By the inductive assumption, $\{B\}_{\Delta_B}$, where $(B,\Delta_B)\in I_d$, can be written as
    $$
    \{B\}_{\Delta_B,d}={}_d m_{B,\Delta_B}+\sum_{\substack{(C,\Delta_C)\prec(B,\Delta_B)\\ \vartheta(\Delta_C)=\vartheta(\Delta_B)}}a_{(B,\Delta_B), (_pC,\Delta_C)} {}_{d}m_{_{p}C,\Delta_C},
    $$
    where $a_{(B,\Delta_B), (C,\Delta_C)}$ is independent of $p \gg 0$.
    Then apply $({}_d m_{_pA,\Delta}, -)_d$ to both sides, we have
    \begin{align*}
         ({}_d m_{{}_pA,\Delta}, \{_pB\}_{\Delta_B,d})_d =&({}_d m_{{}_pA,\Delta}, {}_d m_{_pB,\Delta_B})_d\\
         &+\sum_{\substack{(C,\Delta_C)\prec(B,\Delta_B)\\ \vartheta(\Delta_C)=\vartheta(\Delta_B)}}a_{(B,\Delta_B), (C,\Delta_C)}({}_d m_{{}_pA,\Delta}, {}_{d}m_{{}_{p}C,\Delta_C})_d.
    \end{align*}
    Using $ ({}_d m_{_pA,\Delta}, \{_pB\}_{\Delta_B,d})_d=\sum_{i\leq n_p}c_{B,\Delta_B,p, i}v^{i}$ and Lemma \ref{conv mono}, 
    we conclude that $n_p$ and $c_{B,\Delta_B, p, i}$ for $0\leq i\leq \mrm n_p$, are independent of $p\gg 0$. 
    
    Finally, regard $\{_pA\}_{\Delta,d}$ as $b_{{}_pA,\Delta}$, the construction of $b_{p_A,\Delta}$ shows \eqref{cb-mb} holds. 
    \end{proof}
    
    We introduce a new $\mathbb{Q}(v)$-linear map $\Psi_d:\mmuvgln\to {}_{\mbb Q(v)}\MS$ be the composition of the following maps
    $$
    \mmuvgln\xrightarrow{\varrho}\mmuvsln\xrightarrow{\phi_d} {}_{\mbb Q(v)}\MS.
    $$
    In general $\Psi_d\neq \Phi_d$, but by definition one can see that $\Psi_d= \Phi_d$ when restricted to $\mmuvgln\langle d\rangle$. 
    Moreover we have the following commutative diagram.
    \begin{equation}\label{inverse-sys}
    \begin{tikzcd}
                                & \mmuvgln \arrow[d, "\varrho"] \arrow[ldd, "\Psi_{d+n}"'] \arrow[rdd, "\Psi_d"] &   \\
                               & \mmuvsln \arrow[ld, "\phi^{\m}_{d+n}"] \arrow[rd, "\phi^{\m}_d"']                         &   \\
    {}_{\mbb Q(v)}\mathcal{MS}_{n, d+n} \arrow[rr, "{\phi^{\m}_{d+n,d}}"] &                                                                         & {}_{\mbb Q(v)}\MS
    \end{tikzcd}
    \end{equation}

    The commutativity of the left and right triangles follows from the definition. 
    For the lower triangle, it can be seen by using Corollary \ref{t-map f bas} 
    and the definition of $\phi^{\m}_d$. 

\begin{prop}\label{gln-sln}
     Let $(A,\Delta)\in \widetilde{\Xi}$. Then for all sufficiently large integers $p$ we have 
     $$
     \varrho(\{_{p}A\}_{\Delta})=v^{\vartheta(\Delta)}\varrho(\{_{p-1}A\}_{\Delta}).
     $$
\end{prop}
\begin{proof}
    Denote $|A|=d_0$ and $d=d_0+np$.
    First, using
    $$
    \{_pA\}_{\Delta}= [_pA]_{\Delta}+\sum_{(_pA', \Delta')\sqsubset (_pA,\Delta)}\pi_{(_pA,\Delta), (_pA',\Delta')}[{}_pA']_{\Delta'}, 
    $$
    $$
    [_pA]_{\Delta}+ \text{lower-terms}=M_{_pA,\Delta},
    $$
    and $\xi_{-1}^{\m}(M_{_pA,\Delta})=v^{\vartheta(\Delta)}M_{_pA-I, \Delta}$, we have
     $$
    \xi^{\m}_{-1}(\{_pA\}_{\Delta})= v^{\vartheta(\Delta)}\{_{p-1}A\}_{\Delta}+\sum_{(B, \Delta_B)\in \Xi_{n|n, d-n}}f_{B, \Delta_B}\{B\}_{\Delta_B}.
    $$

    Next, Lemma \ref{cb-gln to schur} shows that for large enough $p$, we have
    \begin{align*}
        \Phi_d(\{{}_pA\}_{\Delta})&=\{{}_pA\}_{\Delta, d},\\
          \Phi_{d-n}(\{{}_{p-1}A\}_{\Delta})&=\{{}_{p-1}A\}_{\Delta, d-n}.
    \end{align*}
    
    By Lemma \ref{t-map f bas} and diagram \eqref{inverse-sys}, we have the following commutative diagram.
    
    \begin{center}
           \begin{tikzcd}
    \mmuvgln \arrow[rr, "\xi^{\m}_{-1}"] \arrow[d, "\Phi_d"'] &  & \mmuvgln \arrow[d, "\Phi_{d-n}"] \\
    \MS \arrow[rr, "{\phi^{\m}_{d, d-n}}"]               &  & \mathcal{MS}_{n, d-n}                  
    \end{tikzcd}
    \end{center}

    Therefore we have
    \begin{align*}
        v^{\vartheta(\Delta)}\{_{p-1}A\}_{\Delta, d-n}=&\phi^{\m}_{d, d-n}\Phi_d(\{_{p}A\}_{\Delta})\\
        =& \Phi_{d-n}\xi_{-1}^{\m}(\{_{p}A\}_{\Delta})\\
        = & v^{\vartheta(\Delta)}\{_{p-1}A\}_{\Delta, d-n}+\sum_{\substack{(B, \Delta_B)\in \Xi_{n|n, d-n}\\ (B,\Delta_B)\sqsubset(_{p-1}A,\Delta)}}f_{B, \Delta_B}\{B\}_{\Delta_B, d-n}.
    \end{align*}
    It follows that $f_{B,\Delta_B}=0$. Thus, for all large enough integer $p$, 
    $$
    \xi^{\m}_{-1}(\{_pA\}_{\Delta})=v^{\vartheta(\Delta)}\{_{p-1}A\}_{\Delta}.
    $$
    Moreover, we have
    $$
    \varrho(\{_pA\}_{\Delta})=\varrho \xi_{-1}^{\m}(\{_pA\}_{\Delta})=v^{\vartheta(\Delta)}\varrho(\{_{p-1}A\}_{\Delta}).
    $$
\end{proof}

Define an equivalence relation on $\widetilde{\Xi}$ defined by 
$$
(A,\Delta)\thicksim (A',\Delta') \text{ if and only if } A-A'=kI \text{ and } \Delta=\Delta'.
$$
We denote the set of equivalence classes by $\widehat{\Xi}$.

For any $\widehat{(A,\Delta)}\in \widehat{\Xi}$ 
we choose a representative $(A,\Delta)\in \Xi_{n|n, d+np}$. 
For large enough $p$, we define
$b_{\widehat{(A,\Delta)}}=v^{-\vartheta(\Delta)\left\lfloor\frac{|A|}{n}\right\rfloor}\varrho(\{A\}_{\Delta})$, 
where $\lfloor a\rfloor$ denotes the greatest integer less than or equal to \(a\).
If we have $B=A+rI$ then
\begin{align*}
    b_{\widehat{(B,\Delta)}}&=v^{-\vartheta(\Delta)\left\lfloor\frac{|B|}{n}\right\rfloor}\varrho(\{B\}_{\Delta})\\
    &=v^{-\vartheta(\Delta)\left\lfloor\frac{|A|+rn}{n}\right\rfloor}\varrho(\{B\}_{\Delta})\\
    &=v^{-\vartheta(\Delta)\left\lfloor\frac{|A|}{n}\right\rfloor}v^{-r\vartheta(\Delta)}\varrho(\{B\}_{\Delta})\\
\end{align*}
By Proposition \ref{gln-sln}, 
$$
v^{-r\vartheta(\Delta)}\varrho(\{B\}_{\Delta})=\varrho(\{A\}_{\Delta}).
$$
So $b_{\widehat{(A,\Delta)}}$ is a well-defined element in $\mmuvsln$. 

\begin{prop}\label{im of cb-sln}
For $\widehat{(A,\Delta)}\in \widehat{\Xi}$ with representative $(A,\Delta)$ such that $|A|=d=d_0+np$, 
we have $\phi^{\mathfrak{m}}_d(b_{\widehat{(A,\Delta)}})=v^{-\vartheta(\Delta)\lfloor \frac{d}{n}\rfloor}\{A\}_{\Delta, d}$ 
for large enough integer $p$.   
\end{prop}
\begin{proof}
    By definition, we have
    \begin{align*}
        \phi^{\m}_d(b_{\widehat{A,\Delta}})=& \phi^{\m}_d(v^{-\vartheta(\Delta)\lfloor \frac{d}{n}\rfloor}\varrho(\{A\}_{\Delta}))\\
        =& v^{-\vartheta(\Delta)\lfloor \frac{d}{n}\rfloor}\Psi_d (\{A\}_{\Delta})\\
        =&  v^{-\vartheta(\Delta)\lfloor \frac{d}{n}\rfloor}\Phi_d (\{A\}_{\Delta})\\
        =& v^{-\vartheta(\Delta)\lfloor \frac{d}{n}\rfloor} \{A\}_{\Delta,d}.
    \end{align*}
\end{proof}

For $r\in\mathbb N$, we have
\[
[r]=\frac{v^r-v^{-r}}{v-v^{-1}},
\qquad
[r]!=\prod_{s=1}^{r}[s],
\]
with the convention that $[0]!=1$. 

For $1\leq i\leq n-1$ and $r\in\mathbb N$, the $r$-th divided powers
of $E_i$ and $F_i$ are defined by
\[
E_i^{(r)}=\frac{E_i^r}{[r]!},
\qquad
F_i^{(r)}=\frac{F_i^r}{[r]!}.
\]

Let ${}_{\mathcal A}\mmuvsln$ be the $\mathcal A$-subalgebra of
$\mmuvsln$ generated by
\[
1_{\bar\lambda},\qquad
E_i^{(r)}1_{\bar\lambda},\qquad
F_i^{(r)}1_{\bar\lambda},\qquad
L1_{\bar\lambda},
\]
where $\bar\lambda\in \mathbb X$, $1\leq i\leq n-1$, and $r\geq1$.
We refer to ${}_{\mathcal A}\mmuvsln$ as the integral form of $\mmuvsln$.

\begin{thm}\label{cb-sln}
    The set $\mathbf{M\dot B}(\slln)=\{b_{\widehat{(A,\Delta)}}\ |\ \widehat{(A,\Delta)}\in \widehat{\Xi}\}$ forms a basis of ${}_{\mcal A}\mmuvsln$ and therefore forms a basis of $\mmuvsln$. 
\end{thm}
\begin{proof}
    For the linear independence, we only need to check that $\mathbf{M\dot B}(\slln)\cap \mmuvsln\langle \bar{d} \rangle$ is linearly independent for each $\bar{d}\in \mbb Z/n\mbb Z$. 
    Thanks to Proposition \ref{im of cb-sln} and $\{\{_pA\}_{\Delta, d}\}$ forms a basis of mirabolic $q$-Schur algebra, the linear independence is clear.

    Next we show that this set spans the whole algebra.
    Since $\varrho$ is surjective, we have ${}_{\mcal A}\mmuvsln=\varrho({}_{\mcal A}\mmuvgln)$. 
    For any $x\in {}_{\mcal A}\mmuvsln$, we have
    \begin{align*}
         x=&\sum_{(A,\Delta)\in \widetilde{\Xi}}f_{A,\Delta}\varrho (\{A\}_{\Delta})\\
         =& \sum_{(A,\Delta)\in \widetilde{\Xi}}f_{A,\Delta}\varrho\xi^{\m}_{p} (\{A\}_{\Delta})\\
         =& \sum_{(A,\Delta)\in \widetilde{\Xi}}f_{A,\Delta}\varrho (v^{-p\vartheta(\Delta)}\{A+pI\}_{\Delta}+\sum_{(B, \Delta_B)\in \Xi_{n|n, d+np}}f_{B, \Delta_B}\{B\}_{\Delta_B})\\
         =& \sum_{(A,\Delta)\in \widetilde{\Xi}}f_{A,\Delta}(v^{\vartheta(\Delta)\lfloor \frac{|A|}{n}\rfloor}b_{\widehat{A,\Delta}}+\sum_{(B, \Delta_B)\in \Xi_{n|n, d+np}}f_{B, \Delta_B}v^{\vartheta(\Delta_b)\lfloor \frac{|B|}{n}\rfloor}b_{\widehat{B,\Delta_B}})
    \end{align*}
    Therefore, $x$ can be written as an $\mcal{A}$-linear combination of $\mathbf{M\dot B}(\slln)$. 
    Therefore, $\mathbf{M\dot B}(\slln)$ forms a basis of ${}_{\mcal A}\mmuvsln$ and moreover a basis of $\mmuvsln$.
\end{proof}

Recall the natural embedding
$$
\dot{\mathbf U}(\slln)\hookrightarrow\mmuvsln.
$$
Under this embedding and the construction of $\mbf{M\dot B}$, we have the following corollary.
\begin{cor}\label{comparsion of cb}
    The canonical basis $\mbf{\dot B}$ of the modified quantum group $\dot{\mathbf U}(\slln)$ 
    is identified with
    $$
    \left\{b_{\widehat{A,\emptyset}}\ \middle|\ \widehat{(A,\emptyset)}\in\widehat{\Xi} \right\}
    $$
    under the embedding.
\end{cor}

\section{Positivity for the Module and Comodule Structure}
For simplicity, in this section we will denote $\uvsln$ as $U$ and its modified form as $\dot{U}=\dot{U}_v(\slln)$. 
Similarly, we denote $\muvsln$ as $MU$ and its modified form as $M\dot{U}=\mmuvsln$.
\subsection{Positivity of the module action}
    As shown by Rosso in \cite{R18} the canonical bases $\{\{A\}_{\Delta}\mid (A,\Delta)\in \Xi_{n|n, d}\}$ do not necessarily satisfy a positivity condition. 
    More specifically, define projections
	$$p_i :
	\mathscr{X}_{n, d} \times 	\mathscr{X}_{n, d} \times 	\mathscr{X}_{n, d}
	\times \mathbb F_q^d \times \mathbb F_q^d
	\longrightarrow
		\mathscr{X}_{n, d}\times 	\mathscr{X}_{n, d} \times \mathbb F_q^d
	$$
	\begin{align*}
		p_{1}(f, f', f'', w, w') &= (f, f', w), \\
		p_{2}(f, f', f'', w, w') &= (f', f'', w'), \\
		p_3(f, f', f'', w, w') &= (f, f'', w+w').
	\end{align*}
    The convolution of $\MS$ can be interpreted as 
    $$
    h\ast h'=p_{3*}(p_1^*h\otimes p_2^*h'),
    $$
    where $p_i^{*}$ denotes pullback along the projection $p_i$ and $p_{i*}$ denotes the push-forward along $p_i$. 
    It is easy to see that the map $p_3$ is not proper and therefore the Decomposition Theorem cannot be applied to the mirabolic setting.
    
    However, if we restrict ourselves to the case that $w=0$ ($w'=0$ resp.), then $p_3$ is a proper map, 
    and the convolution actually endows $\MS$ with a $\mathcal{S}_{n, d}$-bimodule structure, i.e., 
    for any $(A,\emptyset)$ and $(B,\Delta_B)\in \Xi_{n|n,d}$ and
        $$
        \{A\}_{\emptyset}\ast \{B\}_{\Delta_B}=\sum_{(C, \Delta')\in \Xi_{n|n,d}}c_{(A,\emptyset), (B, \Delta_B)}^{(C, \Delta_C)}\{C\}_{\Delta_C},
        $$
        $$
       \{B\}_{\Delta_B}\ast \{A\}_{\emptyset} =\sum_{(C', \Delta')\in \Xi_{n|n,d}}c_{(B,\Delta_B), (A, \emptyset)}^{(C', \Delta')}\{C'\}_{\Delta_{C'}}.
        $$
    Then the Decomposition Theorem implies
    \begin{prop}\label{pos-cb-schur}
        The coefficients $c_{(A,\emptyset), (B, \Delta_B)}^{(C, \Delta_C)}$ and $c_{(B,\Delta_B), (A, \emptyset)}^{(C', \Delta_{C'})}\in \mathbb{N}[v, v^{-1}]$.
    \end{prop}
    
    By the construction of the canonical basis, we have the following identity:
    \begin{align}\label{bar}
    \overline{[A]_{\Delta}}=\sum_{(A', \Delta')\preceq (A, \Delta)}\gamma_{(A', \Delta'), (A, \Delta)}[A']_{\Delta'}
    \end{align}
    where $\gamma_{(A', \Delta'), (A, \Delta)}\in \mathcal{A}$ and $\gamma_{(A, \Delta), (A, \Delta)}=1$.

     
    Consider the $\dot{U}_{v}(\slln)$-module structure of $\mmuvsln$, we have the following positivity result for $\mathbf{M\dot{B}}$. 
    
    \begin{thm}\label{p-module}
        For any $b\in \mathbf{\dot{B}}$ and $b_{\widehat{(A,\Delta)}}\in \mathbf{M\dot{B}}$, we write
        $$
        b\cdot b_{\widehat{(A,\Delta)}}=\sum_{\widehat{(A',\Delta')}\in \widehat{\Xi}}P^{\widehat{(A',\Delta')}}_{b, \widehat{(A,\Delta)}}b_{\widehat{(A',\Delta')}}.
        $$
        where $P^{\widehat{(A',\Delta')}}_{b, \widehat{(A,\Delta)}}=0$ for all but finitely many $\widehat{(A',\Delta')}$.
        Then we have $P^{\widehat{(A',\Delta')}}_{b, \widehat{(A,\Delta)}}\in \mathbb{N}[v,v^{-1}]$.
    \end{thm}
    \begin{proof}
        Let $\Omega$ be the finite set of $\widehat{(A',\Delta')}\in \widehat{\Xi}$ with $P^{\widehat{(A',\Delta')}}_{b, \widehat{(A,\Delta)}}\neq 0$. 
        By Corollary \ref{comparsion of cb}, we may find $b_{\widehat{A'',\emptyset}}\in \mbf{M\dot B}$ such that $b=b_{\widehat{A'',\emptyset}}$.
        Choose the representatives $(A,\Delta), (A',\Delta'), (A'', \emptyset)$ such that $|A|=|A'|=|A''|=d_0$. 
        By Proposition \ref{im of cb-sln}, we can find some sufficiently large integer $p$ and $d=d_0+np$ such that
        $$
        \phi^{\m}_d(b_{\widehat{A'',\emptyset}})=\{{}_{p}A''\}_{\emptyset, d}, 
        \quad \phi^{\m}_d(b_{\widehat{(A,\Delta)}})=v^{-p\vartheta(\Delta)}\{{}_{p}A\}_{\Delta, d}, 
        \quad \phi^{\m}_d(b_{\widehat{(A',\Delta')}})=v^{-p\vartheta(\Delta')}\{_{p}A'\}_{\Delta', d}.
        $$
        Then we have
        \begin{align*}
        \{{}_{p}A''\}_{\emptyset, d}\ast v^{-p\vartheta(\Delta)}\{{}_{p}A\}_{\Delta, d} &=\phi^{\m}_d(b)\ast \phi^{\m}_d(b_{\widehat{A,\Delta}})\\
        &=\phi^{\m}_d(b\cdot b_{\widehat{A,\Delta}})
            =\sum_{\widehat{(A',\Delta')}\in \widehat{\Xi}}P^{\widehat{(A',\Delta')}}_{b, \widehat{(A,\Delta)}} \phi^{\m}_d(b_{\widehat{(A',\Delta')}})\\
             &= \sum_{(A',\Delta')\in \Xi_{n|n, d}}v^{-p\vartheta(\Delta')}P^{(A',\Delta')}_{b, (A,\Delta)} \{_{p}A'\}_{\Delta', d}.
        \end{align*}
        So 
        $$P^{\widehat{(A',\Delta')}}_{b, \widehat{(A,\Delta)}}=v^{p(\vartheta(\Delta')-\vartheta(\Delta))}c^{({}_pA',\Delta')}_{({}_pA'',\emptyset), ({}_pA,\Delta)}.$$
        By Proposition \ref{pos-cb-schur}, we have $P^{\widehat{(A',\Delta')}}_{b, \widehat{(A,\Delta)}}\in \mbb{N}[v, v^{-1}]$ as desired.
    \end{proof}
    \subsection{Positivity of the comodule action}
    
     Fix a composition $\mathbf{b}\in \Lambda_{n, d}$ and a flag $\tilde{f}\in \mathscr{X}_{n, d}(\mathbf{b})$. 
     Let $P_{\mbf{b}}=\mathrm{Stab}_{G_d}(\tilde f)$. For any $\mathbf{a}\in \Lambda_{n, d}$, we can define the $P_{\mbf{b}}$-action on $\mathscr{X}_{n, d}(\mathbf{a})\times V$ via the inclusion $P_\mathbf{b}\hookrightarrow G_{d}$. Define an embedding
     \begin{align*}
         i_{\fb, \fa, w}:\X(\fa)\times V&\to \X(\fb)\times \X(\fa)\times V\\
         (f, w)&\mapsto (\tilde{f}, f, w)
     \end{align*}
     This map identifies the $P_\fb$-orbits on $\X(\fa)\times V$ with the $G_d$-orbits on the image of $i_{\fb, \fa, w}$. 
     Therefore, the pullback of $\ii$ 
    $$
    \ii^*:\mcal{A}_{G_d}(\X(\fb)\times \X(\fa)\times V)\to \mcal{A}_{P_{\fb}}(\X(\fa)\times V)
    $$
    is an isomorphism of $\mcal{A}$-module.

    Now we fix the triple $(f, f', f'')\in \X(\fb)\times \Xp(\fb')\times \Xpp(\fb'')$ 
    such that $\fb=\fb'+\fb''$, $f'=\pi'(f)$ and $ f''=\pi''(f)$. 
    Set $P_{\fb'}=\mathrm{Stab}_{G_{d'}}(f')$ and $P_{\fb''}=\mathrm{Stab}_{G_{d''}}(f'')$, we have similar isomorphisms
    \begin{align}
        \ip^*:\mcal{A}_{G_d'}(\Xp(\fb')\times \Xp(\fa')\times V')&\to \mcal{A}_{P_{\fb'}}(\Xp(\fa')\times V')\\
        \ipp^*:\mcal{A}_{G_d''}(\Xpp(\fb'')\times \Xpp(\fa''))&\to \mcal{A}_{P_{\fb''}}(\Xpp(\fa''))
    \end{align}

    Define 
    $$X^{+}_{\fa, \fa', \fa'',w}=\{(f, w)\in \X(\fa)\times V\ |\ \pi'(f)\in \Xp(\fa'), \pi''(f)\in \Xpp(\fa''),\pi'(w)\in V'\}$$
    Then we have 
    $$
    \X(\fa)\times V\overset{\iota}{\longleftarrow} X^{+}_{\fa, \fa', \fa'',w}\overset{\pi}{\longrightarrow} \Xpp(\fa'')\times \Xp(\fa')\times V'
    $$
    where $\pi(f, w)=(\pi''(f), \pi'(f), \pi'(w))$ and $\iota$ is the natural inclusion. Therefore, we have the linear map
    $$
    \pi_{!}\iota^*:\mcal{A}_{P_\fb}(\X(\fa)\times V)\to \mcal{A}_{P_{\fb''}\times P_{\fb'}}(\Xpp(\fa'')\times \Xp(\fa')\times V')
    $$
    where $\pi_!$ is defined, for all $(\tilde{f}'', \tilde{f}', w')$, 
    by $$\pi_!(h)(\tilde{f}'', \tilde{f}', w')=\sum_{x\in X^{+}_{\fa, \fa', \fa'',w}, \pi(x)=(\tilde{f}'', \tilde{f}', w')}h(x).$$
    We also have
    $$\mcal{A}_{P_{\fb''}\times P_{\fb'}}(\Xpp(\fa'')\times \Xp(\fa')\times V')\cong \mcal{A}_{P_{\fb''}}(\Xpp(\fa''))\otimes \mcal{A}_{P_{\fb'}}(\Xp(\fa')\times V').$$
    
\begin{lem}\label{hyper}
    We have the commutative diagram
    
    \begin{tikzcd}
    {\mcal{A}_{G_{d}}(\X(\fb)\times \X(\fa)\times V)} \arrow[rrr, "\ii^*"] \arrow[dd, "\tilde{\rho}"] &  &  & {\mcal{A}_{P_{\fb}}(\X(\fa)\times V)} \arrow[dd, "\pi_!\iota^*"] \\
                                     &  &  &                               \\
    {\substack{\mcal{A}_{G_{d''}}(\Xpp(\fb'')\times \Xpp(\fa''))\\
     \otimes\mcal{A}_{G_{d'}}(\Xp(\fb')\times \Xp(\fa')\times V')
    }}  \arrow[rrr, "\ipp^*\otimes \ip^*"]                &  &  & {\mcal{A}_{P_{\fb''}}(\Xpp(\fa''))\otimes \mcal{A}_{P_{\fb'}}(\Xp(\fa')\times V')}         
    \end{tikzcd}
\end{lem}
    \begin{proof}
        For any $h\in \mcal{A}_{G_{d}}(\X(\fb)\times \X(\fa)\times V)$ and $(\tilde{f}'', \tilde{f}', v')\in \Xpp(\fa'')\times \Xp(\fa')\times V'$, we have
        \begin{align*}
            \pi_!\iota^*\ii^*h(\tilde{f}'', \tilde{f}', w')
            =& \sum_{(\tilde{f}, w)\in Z_{\tilde{f}', \tilde{f}'', w'}}\ii^*h(\tilde{f}, w)=\sum_{(\tilde{f}, w)\in Z_{\tilde{f}', \tilde{f}'', w'}}h(f, \tilde{f}, w)\\
            = &\tilde{\rho}_{\fb',\fa', \fb'', \fa'', w}(h)(f'', \tilde{f}'', f', \tilde{f}', w')\\
            = &(\ipp^*\otimes \ip^*)\circ\tilde{\rho}_{\fb',\fa', \fb'', \fa'', w}(h)(\tilde{f}'', \tilde{f}', w')
        \end{align*}
    \end{proof}

    Recall the canonical basis $\{\{A\}_d\ |\ A\in \Theta_d\}$ of $\mcal S_{n, d}$ from \cite{BLM90} and 
    the canonical basis $\{\{A\}_{\Delta,d} |\ (A, \Delta)\in \Xi_{n|n, d}\}$ of $\mcal{MS}_{n, d}$. 
    We write
    $$
    \rho_{\fb', \fa', \fb'', \fa''}(\{A\}_{\Delta,d})=\sum_{(A',\Delta')\in \Xi_{n|n, d'}, A''\in \Theta_{d''}}c^{ A'',(A',\Delta')}_{A,\Delta}\{A''\}_{d''}\otimes \{A'\}_{\Delta', d'},\quad c^{ A'',(A',\Delta')}_{A,\Delta}\in \mcal A
    $$

    \begin{prop}\label{pos-cmod}
     The coefficient $c^{ A'',(A',\Delta')}_{A,\Delta}\in \mathbb{N}[v, v^{-1}]$.   
    \end{prop}
    \begin{proof}
        It is sufficient to show the proposition for $\tilde{\rho}_{\fb', \fa', \fb'', \fa''}$. 
        By Lemma \ref{hyper} and the definition of canonical basis, we only need to show that the linear map $\pi_!\iota^*$ is the function version of the hyperbolic localization functor introduced in \cite{B03}.
        
        To do this, we now consider the algebraic closure $\overline{\mathbb F}_{q}$ of $\mathbb F_{q}$. 
        Let $\bar{G}_d$ be the general linear group over $\overline{\mathbb F}_{q}$ whose $\mathbb F_{q}$-points are $G_d$.
        Similarly, we define the variety $\overline{\mathscr{X}}_{n, d}(\fa)$ over $\overline{\mathbb F}_{q}$ for $\X(\fa)$. 
        Take $\bar{G}_m=\overline{\mathbb F}_{q}^{*}$.
        
        For $d=d'+d''$, we fix an isomorphism $\overline{\mathbb F}_{q}^d\cong \overline{\mathbb F}_{q}^{d'}\oplus \overline{\mathbb F}_{q}^{d''}$. 
        Then we can fix an embedding 
        $$
        \bar{G}_m\to \bar{G}_d,\, t\mapsto (1_{\overline{\mathbb F}_{q}^{d'}}, t1_{\overline{\mathbb F}_{q}^{d''}}).
        $$
        We define the action of $\bar{G}_m$ on $\overline{\mathscr{X}}_{n, d}(\fa)$ and $\overline{\mathbb F}_{q}^d$ by this embedding
        and these actions induce the diagonal action of $\bar{G}_m$ on $\overline{\mathscr{X}}_{n, d}(\fa)\times \overline{\mathbb F}_{q}^d$. 
        By \cite[Proposition 1.5.3]{FL21}, we know the set of fixed points of $\overline{\mathscr{X}}_{n, d}(\fa)$ under this action is
        $$
        \bigsqcup_{\fa'+\fa''=\fa}\overline{\mathscr{X}}_{n, d''}(\fa'')\times \overline{\mathscr{X}}_{n, d'}(\fa').
        $$
        So we can obtain fixed-points set of $\overline{\mathscr{X}}_{n, d}(\fa)\times \overline{\mathbb F}_{q}^{d'}$ is 
        $$
        \bigsqcup_{\fa'+\fa''=\fa}\overline{\mathscr{X}}_{n, d''}(\fa'')\times \overline{\mathscr{X}}_{n, d'}(\fa')\times \overline{\mathbb F}_{q}^{d'}.
        $$
        For any $v\in \overline{\mathbb F}_{q}^{d}$, we can see $\lim_{t\to 0}t\cdot v\in \overline{\mathbb F}_{q}^{d'}$, then by \cite[Proposition 1.5.3]{FL21}, 
        the attracting set of $\overline{\mathscr{X}}_{n, d}(\fa)\times \overline{\mathbb F}_{q}^{d}$ is the algebraic variety whose $\mathbb{F}_q$-points are $X^{+}_{\fa, \fa', \fa'',v}$.

        Therefore, the linear map $\pi_!\iota^*$ is the function version of the hyperbolic localization functor attached to the data $(\overline{\mathscr{X}}_{n, d}(\fa)\times \overline{\mathbb F}_{q}^{d}, \bar{G}_m)$. By \cite{B03}, this functor sends a simple perverse sheaf to a semisimple complex. Recall that $\ii^*(\{A\}_{\Delta,d})$ is the function version of simple perverse sheaf attached to the $P_{\fb}$-orbit in $\overline{\mathscr{X}}_{n, d}(\fa)\times \overline{\mathbb F}_{q}^{d}$ indexed by $(A,\Delta)$. 
        Hence the hyperbolic localization of the IC sheaf corresponding to $(A,\Delta)$ is a semisimple complex. By the sheaf-function
        correspondence, its trace function is a linear combination of canonical basis elements with coefficients in $\mathbb N[v,v^{-1}]$. 
        Using Lemma~\ref{hyper}, this trace function is exactly the image of $\{A\}_{\Delta,d}$ under $\tilde\rho_{\fb',\fa',\fb'',\fa''}$.
        Thus the proposition follows.
    \end{proof}
    Due to the definition of transfer map $ \phi^\m_{d+n, d}$, we may consider the positivity of it, i.e., we have the following result.

    \begin{prop}\label{pos-trans}
    The transfer map $\phi^\m_{d+n, d}: \mathcal{MS}_{n, d+n}\to \MS$ sends each canonical basis to a linear combination of canonical basis elements with coefficients in $\mbb N[v, v^{-1}]$ or zero.
    \end{prop}
    \begin{proof}
    Recall that $ \phi^\m_{d+n, d}=(\chi \otimes \mathrm{id})\rho$. 
    Li and Wang prove the positivity of $\chi$ in \cite[Proposition 3.3]{LW18}. 
    The positivity of $\rho$ is ensured by Proposition \ref{pos-cmod}. 
    This proves the proposition
    \end{proof}

    It is known that the algebra $U$ is a Hopf algebra, with comultiplication $\Delta:U\to U\otimes U$ given on the generators by 
    \[\Delta(E_i) = E_i\otimes K_i+ 1\otimes E_i ,\quad\Delta(F_i) = F_i\otimes 1+K_i^{-1}\otimes F_i ,\quad\Delta(K_i^{\pm 1}) = K_i^{\pm 1}\otimes K_i^{\pm 1}.\]

Goyal and Rosso endowed $MU$ with a $U$-comodule algebra structure \cite{GR25}:
\begin{align*}
     \rho(K_{i}^{\pm 1})&=\Delta(K_{i}^{\pm 1})=K_{i}^{\pm 1}\otimes K_{i}^{\pm 1};\\
     \rho(E_{i})&=\Delta(E_{i})=E_{i}\otimes K_{i}+1\otimes E_i;\\
      \rho(F_{i})&=\Delta(F_{i})=F_{i}\otimes 1+K_{i}^{-1}\otimes F_i;\\
       \rho(L)&=1\otimes L.
\end{align*}

Recall the algebra homomorphism $\Phi_d:\muvgln\to {}_{\mbb{Q}(v)}\MS$ defined in \cite{FZM25}. 
This induces a surjective algebra homomorphism (still denoted by $\Phi_d$) $\Phi_d:\muvsln\to {}_{\mbb{Q}(v)}\MS$ 
by sending $E_i\to E_i$, $F_i\to F_i$, $L\to L$ and $K_i\to H_iH^{-1}_{i+1}$ for any $i\in [1, n-1]$. 
For $d=d'+d''$ we have the following commutative diagram.
\begin{center}
    \begin{tikzcd}
MU \arrow[d, "\Phi_d"'] \arrow[rr, "\rho"] &  & U\otimes MU \arrow[d, "\kappa_{d''}\otimes \Phi_{d'}"] \\
{}_{\mbb{Q}(v)}\MS \arrow[rr, "\rho"]                        &  & {}_{\mbb{Q}(v)}\Spp \otimes {}_{\mbb{Q}(v)}\MSp                                      
\end{tikzcd}
\end{center}
where $\kappa_{d''}:U\to \qSpp$ is the surjective algebra homomorphism from \cite{BLM90}.

Let $I=\{1,\cdots, n-1\}$. Recall that $U$ is a $\mbb Z[I]$-graded algebra by setting 
$$\deg E_i=(0,\ldots,0,\stackrel{i}{1},0,\ldots,0),\quad \deg F_i=(0,\ldots,0,\stackrel{i}{-1},0,\ldots,0),$$
$$\deg K_i=\deg K_i^{-1}=(0,\ldots,0)$$
for all $1\leq i\leq n-1$, i.e., 
$$
U=\bigoplus_{\nu\in \mbb Z[I]}U(\nu).
$$
In \cite{GR25}, Goyal and Rosso also show that the algebra $MU$ is $\mathbb{Z}[I]$-graded by further setting $\deg L=(0,\ldots,0)$.
Then we have 
$$
MU=\bigoplus_{\nu\in \mbb Z[I]}MU(\nu).
$$
Define an injective map $I\to \mbb X$ by sending $i\mapsto \bar{s}_i-\bar{s}_{i+1}$, where $s_i$ is the $i$-th standard basis element in $\mbb Z^{n}$.
Then $\mbb Z[I]$ can be seen as a subset of $\mbb X$ by this injective map.

For any triple $\nu, \nu', \nu''\in \mbb Z[I]$ such that $\nu=\nu'+\nu''$, we obtain
$$
\rho_{\nu', \nu''}: MU(\nu)\to U(\nu'')\otimes MU(\nu')
$$
by restricting $\rho$ to $MU(\nu)$ and then project to $U(\nu'')\otimes MU(\nu')$.
Moreover, restricting $\Phi_d$ to $MU(\nu)$ induces a linear map which will still be denoted by $\Phi_d$:
$$
\Phi_d: MU(\nu)\to \bigoplus_{\bar{\fb}-\bar{\fa}=\nu}\qMS(\fb,\fa).
$$
Similarly, we have
$$
\kappa_d: U(\nu)\to \bigoplus_{\bar{\fb}-\bar{\fa}=\nu}\qS(\fb,\fa).
$$

For $\nu=\nu'+\nu''$, taking the component gives the following commutative diagram.
\begin{center}
    \begin{tikzcd}
MU(\nu) \arrow[d, "\kappa^{\m}_d"'] \arrow[rr, "\rho_{\nu', \nu''}"] &  & U(\nu'')\otimes MU(\nu') \arrow[d, "\kappa_{d''}\otimes \kappa^{\m}_{d'}"] \\
\bigoplus_{\bar{\fb}-\bar{\fa}=\nu}\qMS(\fb,\fa) \arrow[rr, "\oplus \rho_{\fb',\fa',\fb'',\fa''}"]                        &  &\bigoplus^{\substack{\fb=\fb'+\fb'', \\ \fa=\fa'+\fa''}}_{\substack{\bar{\fb''}-\bar{\fa''}=\nu'', \\ \bar{\fb'}-\bar{\fa'}=\nu'}}\qSpp(\fb'',\fa'') \otimes \qMSp(\fb',\fa')            
\end{tikzcd}
\end{center}

Following \cite[23.1.5]{L93}, the comodule structure on $MU$ induces a $\dot{U}$-comodule structure on $M\dot{U}$. 
For any $\lam',\muu', \lam'',\muu''\in \mbb X$, let $\lam=\lam'+\lam''$ and $\muu=\muu'+\muu''$, there is a well-defined linear map 
\begin{align*}
\rho_{\lam',\muu', \lam'',\muu''}: {}_{\lam}MU_{\muu}\to {}_{\lam''}U_{\muu''} \otimes {}_{\lam'}MU_{\muu'}
\end{align*}
such that
\begin{align*}
\rho_{\lam',\muu', \lam'',\muu''}(\pi^{\m}_{\lam, \muu}(x))=(\pi_{\lam'', \muu''}\otimes \pi^{\m}_{\lam', \muu'})(\rho(x))
\end{align*}
for all $x\in MU$.

Therefore, for any $\lam-\muu=\nu$,  $\lam'-\muu'=\nu'$ and $\lam''-\muu''=\nu''$ 
we have the following commutative diagram when passing to the corresponding modified components.

\begin{center}
    \begin{tikzcd}
MU(\nu) \arrow[d, "\pi^{\m}_{\lam, \muu}"] \arrow[rr, "\rho_{\nu', \nu''}"] &  & U(\nu'')\otimes MU(\nu') \arrow[d, "\pi_{\lam'',\muu''}\otimes \pi^{\m}_{\lam',\muu'}"] \\
{}_{\lam}MU_{\muu} \arrow[rr, " \rho_{\lam',\muu',\lam'',\muu''}"]                        &  &{}_{\lam''}U_{\muu''} \otimes {}_{\lam'}MU_{\muu'} 
\end{tikzcd}
\end{center}

The main result of this section is to prove that the coefficient of the linear map $\rho_{\lam',\muu', \lam'',\muu''}$ is in $\mbb N[v, v^{-1}]$.

We consider that the homomorphism $\phi^{\m}_d$ induces a linear map that we still denote by $\phi^{\m}_d$:
$$
\phi^{\m}_d:{}_{\fbb}MU_{\faa}\to \qMS(\fb, \fa).
$$
Then by definition, if we take $\lam$ and $\muu\in \mbb X$ such that $\lam-\muu=\nu$, we obtain a commutative diagram
\begin{center}
    \begin{tikzcd}
MU(\nu) \arrow[d, "\kappa^{\m}_d"'] \arrow[rr, "\pi^{\m}_{\lam, \muu}"] &  & {}_{\lam}MU_{\muu} \arrow[d, "\phi^{\m}_{d}"] \\
\bigoplus_{\fbb-\faa=\nu}\qMS(\fb,\fa) \arrow[rr, " "]                        &  &\qMS(\lambda, \mu)
\end{tikzcd}
\end{center}

Finally, we combine the above four commutative diagrams and use a diagram chasing, we have the following commutative diagram.
\begin{center}
    \begin{tikzcd}
{}_{\lam}MU_{\muu}  \arrow[d, "\phi^{\m}_d"'] \arrow[rr, "\rho_{\lam',\muu',\lam'',\muu''}"] &  & {}_{\lam''}U_{\muu''} \otimes {}_{\lam'}MU_{\muu'} \arrow[d, "\phi_{d''}\otimes \phi^{\m}_{d'}"] \\
\qMS(\lambda,\mu) \arrow[rr, "\rho_{\lambda',\mu',\lambda'',\mu''}"]                        &  &\qSpp(\lambda'',\mu'') \otimes \qMSp(\lambda',\mu')            
\end{tikzcd}
\end{center}

\begin{thm}\label{pos-cmod-sln}
    Let $b_{\widehat{(A,\Delta)}}\in \mbf{M\dot{B}}$, we write
    $$
    \rho_{\lam',\muu',\lam'',\muu''}(b_{\widehat{(A,\Delta)}})=\sum_{b''\in \mathbf{\dot{B}}, b_{\widehat{(A',\Delta')}}\in \mathbf{M\dot{B}}} m_{\widehat{(A,\Delta)}}^{b'',\widehat{(A',\Delta')}}b''\otimes b_{\widehat{(A',\Delta')}}.
    $$
    Then $m_{\widehat{(A,\Delta)}}^{b'',\widehat{(A',\Delta')}}\in \mbb N[v, v^{-1}]$.
\end{thm}
\begin{proof}
    Let $\Omega=\{(b'', b_{\widehat{(A',\Delta')}})\ |\ m_{\widehat{(A,\Delta)}}^{b'',\widehat{(A',\Delta')}}\neq 0 \}$, it is clear that $\sharp \Omega<\infty$.
    For $\widehat{(A,\Delta)}$, $\widehat{(A',\Delta')}\in \widehat{\Xi}$,  Proposition \ref{im of cb-sln} shows, we can find some large enough integers $d, d', d''$ with $d=d'+d''$ 
    such that for any $(b'', b_{\widehat{(A',\Delta')}})\in \Omega$
    $$
    \phi_{d''}(b'')=\{B\}_{d''}, \quad \phi^{\m}_d(b_{\widehat{(A,\Delta)}})=v^{-\vartheta(\Delta)\lfloor \frac{d}{n}\rfloor}\{A\}_{\Delta,d}, \quad \phi^{\m}_{d'}(b_{\widehat{(A',\Delta')}})=v^{-\vartheta(\Delta')\lfloor \frac{d'}{n}\rfloor}\{A'\}_{\Delta',d'}.
    $$
    where $\{A\}_{\Delta,d}$, $\{A'\}_{\Delta',d'}$ are certain canonical basis elements in the mirabolic $q$-Schur algebra and $\{B\}_{d''}$ is a certain canonical basis element in $\Spp$.
    Then we have
    \begin{align*}
     \phi_{d''}  \otimes \phi^{\m}_{d'}(\rho_{\lam',\muu',\lam'',\muu''}(b_{\widehat{(A,\Delta)}})=&\sum_{(b'',b_{\widehat{(A',\Delta')}})\in \Omega} 
    m_{\widehat{(A,\Delta)}}^{b'',\widehat{(A',\Delta')}}\phi_{d''}(b'')\otimes \phi^{\m}_{d'}(b_{\widehat{(A',\Delta')}})\\
    =& \sum_{(b'',b_{\widehat{(A',\Delta')}})\in \Omega} 
    m_{\widehat{(A,\Delta)}}^{b'',\widehat{(A',\Delta')}}\{B\}_{d''}\otimes v^{-\vartheta(\Delta')\lfloor \frac{d'}{n}\rfloor}\{A'\}_{\Delta',d'}\\
    =&v^{-\vartheta(\Delta)\lfloor \frac{d'}{n}\rfloor}\rho_{\lambda',\mu',\lambda'',\mu''}(\{A\}_{\Delta,d})\\
    \end{align*}
    Then Proposition \ref{pos-cmod} shows that $m_{\widehat{(A,\Delta)}}^{b'',\widehat{(A',\Delta')}}\in \mbb N[v, v^{-1}]$.
\end{proof}

With this theorem, we have a direct corollary.
\begin{cor}
    The algebra homomorphism $\phi^{\m}_d:\mmuvsln\to \MS$ sends each canonical basis element to a linear combinations of canonical basis elements with coefficients in $\mbb N[v, v^{-1}]$ or zero.
\end{cor}
\begin{proof}
    Take $b_{\widehat{(A,\Delta)}}\in \mathbf{M\dot{B}}$, we may assume $b_{\widehat{(A,\Delta)}}\in \mathbf{M\dot{B}}\langle\bar{d}\rangle$, where $ \mathbf{M\dot{B}}\langle\bar{d}\rangle$ is the basis of $\mmuvsln\langle \bar{d}\rangle$, otherwise we have $\phi^{\m}_d(b_{\widehat{(A,\Delta)}})=0$. 
    By Proposition \ref{im of cb-sln}, we have $\phi^{\m}_{d+np}(b_{\widehat{(A,\Delta)}})$ is a canonical basis element in $\mathcal{MS}_{n,d+np}$ times a power of $v$ for large enough $p$.
    Then using \eqref{inverse-sys}, we have
    $$
    \phi^{\m}_d(b_{\widehat{(A,\Delta)}})=\phi^{\m}_{d+n, d}\circ \cdots \circ \phi^{\m}_{d+np, d+np-n}(\phi^{\m}_{d+np}(b_{\widehat{(A,\Delta)}})).
    $$
    Then applying Proposition \ref{pos-trans} and Proposition \ref{im of cb-sln}, the corollary follows.
\end{proof}

\bibliographystyle{alpha}
\bibliography{positivity}
\end{document}